\documentclass[oneside]{amsart}
\usepackage{geometry}
 \usepackage{graphicx}
 \usepackage{amssymb}
 \usepackage[mathscr]{euscript}
\usepackage{amsmath}
\usepackage{amsfonts}
\usepackage{mathtools}
\usepackage{amsthm}
\usepackage{hyperref}
\usepackage{enumerate}
\usepackage{indentfirst}
\usepackage{tikz}
\usepackage{comment}
\usetikzlibrary{arrows,automata,positioning}
\numberwithin{equation}{section}
\allowdisplaybreaks

\newtheorem{thm}{Theorem}
\newtheorem{prop}[thm]{Proposition}

\newtheorem{lem}[thm]{Lemma}
\newtheorem{fact}[thm]{Fact}

\newtheorem{cor}[thm]{Corollary}
\newtheorem{ques}[thm]{Question}
\newtheorem*{thm*}{Theorem}
\newtheorem*{thmA}{Theorem A}
\newtheorem*{thmB}{Theorem B}

\numberwithin{thm}{section}

\theoremstyle{definition}
\newtheorem{defn}[thm]{Definition}
\newtheorem{ex}[thm]{Example}

\newcommand{\N}{\mathbb{N}}
\newcommand{\R}{\mathbb{R}}

\newcommand{\Z}{\mathbb{Z}}
\newcommand{\Q}{\mathbb{Q}}

\newcommand{\card}{\operatorname{card}}

\author{Philipp Hieronymi}
\address{Mathematical Institute\\ University of Bonn\\
Endenicher Allee 60\\ 53115 Bonn\\ Germany}
\email{hieronymi@math.uni-bonn.de}

\author{Chris Schulz}
\address{Department of Mathematics\\University of Illinois at Urbana-Champaign\\1409 West Green Street\\Urbana, IL 61801\\ USA}
\address{Department of Pure Mathematics\\
200 University Avenue West\\
Waterloo, Ontario\\
N2L 3G1\\
Canada
}
\email{chris.schulz@uwaterloo.ca}

\title{A strong version of Cobham's theorem}
\usepackage[a-2b,mathxmp]{pdfx}[2018/12/22]

\begin{document}

\maketitle

\begin{abstract}
 Let $k,\ell\geq 2$ be two multiplicatively independent integers. Cobham's famous theorem states that a set $X\subseteq \mathbb{N}$ is both 
$k$-recognizable and $\ell$-recognizable if and only if it is definable in Presburger arithmetic. Here we show the following strengthening:
let $X\subseteq \N^m$ be $k$-recognizable, let $Y\subseteq \N^n$ be $\ell$-recognizable such that both $X$ and $Y$ are not definable in Presburger arithmetic. Then the first-order logical theory of $(\mathbb{N},+,X,Y)$ is undecidable. This is in contrast to a well-known theorem of B\"uchi that the first-order logical theory of $(\mathbb{N},+,X)$ is decidable.
\end{abstract}

\section{Introduction}

Let $k,\ell \in \N_{\geq 2}$. We say $k$ and $\ell$ are \textbf{multiplicatively independent} if $k^a\neq \ell^b$ for all nonzero $a,b\in \Z$. 
A subset $X$ of $\N^n$ is \textbf{$k$-recognizable} if the set of $k$-ary representations of $X$ is recognized by some finite automaton. We refer the reader to Allouche and Shallit \cite{AS} for details. \textbf{Cobham's theorem} \cite{Cobham} states that a subset of $\N$ can only be both $k$- and $\ell$-recognizable for multiplicatively independent $k$ and $\ell$ if it is definable in \textbf{Presburger arithmetic}, that is first-order definable in $(\N,+)$ (or equivalently: a finite union of arithmetic progressions). This result was extended to subsets of $\N^n$ by Sem\"enov \cite{Semenov}. For surveys on Cobham's theorem, see Durand and Rigo \cite{durandrigo} and Bruy\`ere et al. \cite{BHMV}. Subsets of $\N^n$ definable in Presburger arithmetic are well understood: by Ginsburg and Spanier \cite{GinsSpanier} a subset $X\subseteq \N^m$ is first-order definable in $(\N,+)$ if and only if $X$ is semilinear.  In particular, the order relation on $\N$ is first-order definable in $(\N,+)$\footnote{Indeed, the order relation is first-order definable in $(\N,+)$ using the first-order formula $\exists z \ x+z =y$.}. For a recent survey on Presburger arithmetic, see Haase \cite{Hasse-Survey}.  As noted in \cite{durandrigo}, the essence of Cobham's theorem is that recognizability depends strongly on the choice of the base $k$. Here we strengthen this conclusion of the Cobham-Sem\"enov theorem as follows: Given two subsets of some Cartesian power of $\N$ such that both are not first-order definable in $(\N,+)$, but one is $k$- and the other one is $\ell$-recognizable, then they are not only distinct, but the resulting theory of the two sets over Presburger arithmetic is undedicable.
\begin{thmA}
 Let $k,\ell\in \N_{\geq 2}$ be multiplicatively independent. Let $X\subseteq \N^m$ be $k$-recognizable and let $Y\subseteq \N^n$ be $\ell$-recognizable such that both $X$ and $Y$ are not first-order definable in $(\N,+)$. Then the first-order logical theory ${\rm FO}(\mathbb{N},+,X,Y)$ is undecidable.
\end{thmA}

\noindent First, we give more details about the connection of Theorem A to earlier work and the Cobham-Sem\"enov theorem. Throughout, when we write \emph{definable} without further qualifications, we mean \emph{first-order definable without parameters}. For $k\in \N_{\geq 2}$, let $k^{\N}$ denote $\{k^i : i\in \N\}$, and let $V_k : \N \to k^{\N}$ be the function that maps $0$ to $1$ and maps $x\in \N_{\geq 1}$ to the largest power of $k$ that divides $x$. We consider the structure $(\N,+,V_k)$, often called \textbf{B\"uchi arithmetic of base $k$}. The following result is due
to B\"uchi \cite{Buechi60}, although the original statement was incorrect as pointed out by MacNaughton \cite{mcnaughton_1963}.

\begin{fact}[{see \cite[Theorem 6.1 \& Corollary 6.2]{BHMV}}]\label{fact:BHMV}
Let $k\geq 2$ and $X\subseteq \N^n$. Then $X$ is $k$-recognizable if and only if $X$ is definable in $(\N,+,V_k)$. Moreover, ${\rm FO}(\N,+,V_k)$ is decidable.
\end{fact}

\noindent  We refer the reader to \cite{BHMV} for a detailed discussions of this connection between recognizability and first-order definability. Decidability fails for more expressive theories. By Villemaire \cite{Villemaire}, when $k,\ell$ are multiplicatively independent, then the theory ${\rm FO}(\N,+,V_k,V_{\ell})$ is undecidable. B\`es \cite{Bes97} strengthened this result by showing that under the same assumptions on $k$ and $\ell$, the theory ${\rm FO}(\N,+,V_k,L)$ is already undecidable whenever $L\subseteq \N^n$ is $\ell$-recognizable, but not definable in $(\N,+)$. The work in \cite{Bes97} further establishes that under the above assumptions the structure $(\N,+,V_k,L)$ (and hence also $(\N,+,V_k,V_{\ell})$) defines multiplication and thus every arithmetic set. \newline

\noindent We already claimed that Theorem A is a strengthening of the Cobham-Sem\"enov theorem. Indeed, by the following argument which first appeared in \cite[Second proof of Theorem 5.1]{Bes97}, the latter can be deduced from the former. Suppose that $X\subseteq \N^n$ is both $k$- and $\ell$-recognizable where $k$ and $\ell$ are multiplicatively independent, but not definable in $(\N,+)$. Then ${\rm FO}(\N,+,X,X)$ is undecidable by Theorem A. However, ${\rm FO}(\N,+,X)$ is decidable by Fact \ref{fact:BHMV}, a contradiction\footnote{Of course, the main result from \cite{Bes97} can be used directly:  Theorem A obviously generalizes Bes' undecidability result, which is already known to imply the Cobham-Sem\"emov theorem.}.\newline

\noindent The proof of Theorem A reduces to the case of very simple $k$- and $\ell$-recognizable sets by the following result from \cite{Bes97}.

\begin{fact}[{\cite[Theorem 3.1]{Bes97}}]\label{fact:bes} Let $k\geq 2$. If $X\subseteq \N^n$ is $k$-recognizable and not definable in $(\N,+)$, then $k^{\N}$ is definable in $(\N,+,X)$.
\end{fact}

\noindent Thus in order to establish Theorem A, it is enough to prove the following.

\begin{thmB}
Let $k,\ell\in \N_{\geq 2}$ be multiplicatively independent. Then the theory ${\rm FO}(\mathbb{N},+,k^{\N},\ell^{\N})$ is undecidable.
\end{thmB}

\noindent The undecidability of ${\rm FO}(\mathbb{N},+,k^{\N},\ell^{\N})$, even just the case $k=2$ and $\ell=3$ was a long-standing open problem mentioned in the literature as early as the mid 1980s (see the survey articles \cite{Bes-Survey,MV-Survey,point2010expansion}), and has sometimes been attributed to Lou van den Dries. It is worth mentioning that even ${\rm FO}(\N,+,2^{\N},3^{\N})$ includes many non-trivial number-theoretic statements about $2$ and $3$.  An important corollary of Baker's famous theorem on linear forms \cite{Baker} is that for every $m\in \N$, there exists $C(m)\in \N$ such that if $n_1,n_2 \in \N$  with $2^{n_1}-3^{n_2}=m$, then $n_1,n_2\leq C$. This statement is equivalent to the truth of the following first-order sentence in $(\N,+,2^{\N},3^{\N})$:
\[
\forall u \in \N \exists v \in \N  \forall x \in k^{\N} \forall y \in \ell^{\N} \ (x \geq v \vee y \geq v) \rightarrow |x-y|>u.
\]

\noindent In contrast to earlier work in \cite{Villemaire,Bes97}, we prove Theorem B by encoding the halting problem in ${\rm FO}(\mathbb{N},+,k^{\N},\ell^{\N})$ rather than by defining multiplication in $(\mathbb{N},+,k^{\N},\ell^{\N})$. Indeed, the second author showed in \cite{S-undefinability}  using arguments from Friedman and Miller \cite{FM} that the latter strategy fails and  $(\mathbb{N},+,k^{\N},\ell^{\N})$ does not define multiplication.\newline

\noindent It is worth outlining why ${\rm FO}(\mathbb{N},+,k^{\N},\ell^{\N})$ is expressive enough to encode the halting problem. Let $k,\ell\in \N_{>1}$ be multiplicatively independent. For $x\in k^{\N}$, we define $S(x)\subseteq \ell^{\N}$ as the set of all $y\in \ell^{\N}$ such that $x$ is the second largest power of $k$ that appears with a non-zero digit in the $k$-ary representation of $y$. Just as an example, consider $k=2$ and $\ell=3$. In this case $\ell^3=k^4+k^3+k+1$, and thus $\ell^3 \in S(k^3).$ It is rather easy to show that $S(x)$ is uniformly definable in $(\N,+,k^{\N},\ell^{\N})$; that means, not only is $S(x)$ definable in $(\N,+,k^{\N},\ell^{\N})$, but the formula used to define $S(x)$ does not depend on $x$ and uses $x$ only as a parameter.  It follows from the aforementioned corollary of Baker's theorem that each $S(x)$ is finite. However, while each $S(x)$ is finite, there is no uniform bound. Indeed, we establish that we can find every possible pattern in tuples of subsets of $\ell^{\N}$ of the form $(S(k^m),S(k^{m+1}),\dots,S(k^{m+n}))$ (see Lemma \ref{lem:pickoutvalues} for a precise statement). Using this genericity of tuples and the fact that $S(x)$ is uniformly definable in $(\N,+,k^{\N},\ell^{\N})$, we are able to code certain statements about the existence of tuples of finite subsets of $\N^2$ in $(\N,+,k^{\N},\ell^{\N})$. Following the proof of the undecidability of the weak monadic second order theory of the grid\footnote{The grid is the structure $(\N^2,s_0,s_1)$, where $s_0(m,n):=(m+1,n)$ and $s_1(m,n):=(m,n+1)$.} (see Thomas \cite[Theorem 16.5]{Thomas}), we show that this is enough to code the halting problem in ${\rm FO}(\mathbb{N},+,k^{\N},\ell^{\N})$.\newline

\noindent A previous version of this paper appeared in the \emph{Proceedings of the 54th Annual ACM SIGACT Symposium on Theory of Computing (2022)}.

\subsection*{Acknowledgements} The authors were partially supported by NSF grant DMS-1654725. The first author was partially supported by the \emph{Hausdorff Center for Mathematics} at the University of Bonn. We thank Chris Miller, Jeffrey Shallit and Erik Walsberg for helpful conversations around the topic of this paper.

\subsection*{Notation} The set $\N$ contains $0$. Throughout, definable  means first-order definable without parameters. We use $i,j,k,\ell,m,n$ to denote elements of $\N$. For a set $X$, we use $\card(X)$ to denote the cardinality of $X$, and we use $\mathcal{P}(X)$ to denote the power set of $X$.

\section{A consequence of density}

It is well-known that $k^{\N}\ell^{-\N}$ is dense in $\R_{>0}$ by Kronecker's Approximation Theorem (see \cite[Theorem 7.8]{Apostol}). 
Our proof of Theorem B relies heavily on the observation that certain consequences of this density can be coded in $\mathrm{FO}(\mathbb{N},+,k^\N, \ell^\N)$. In this section, we establish one such consequence that plays a crucial role later.
Before we state this consequence, we recall (and prove) the following basic result from real analysis.

\begin{lem}\label{lem:intervals}
Let $I=[a_0,b_0)$ be an interval in $\R$, and let $(a_n)_{n\in \N_{>0}}$, $(b_n)_{n\in \N_{>0}}$ be two sequences of real numbers such that $\{a_n \ : \ n \in \N_{>0}\}$ is dense in $I$ and for every $n\in \N$
\begin{enumerate}
    \item $a_n<b_n$,
    \item $(b_{n+1} - a_{n+1}) < \frac{1}{4} (b_n - a_n)$.
\end{enumerate}
Then there exists $m\in \N_{>0}$ such that $[a_m,b_m)\subseteq I$ and
\[
[a_m,b_m) \cap \bigcup_{n=1}^{m-1} [a_n,b_n) = \emptyset. 
\]
\end{lem}
	\begin{proof}
	Set $I_n:=[a_n,b_n)$ for each $n\in \N_{\geq 1}$ and set  $J:=I \cap (\bigcup_{n=1}^{\infty} I_n)$. For $X\subseteq \R$, we write $\mathfrak{L}(X)$ for the Lebesgue measure of $X$.
	We now observe that by (2)
	\[
	\mathfrak{L}(J) \leq (b_0-a_0) \sum_{n\in \N_{> 0}} 4^{-i} = \frac{1}{3}(b_0-a_0) < \mathfrak{L}(I).
	\]
	Thus $I\setminus J$ is infinite. Let $x_1,x_2\in I\setminus J$ be such that $x_1 < x_2$. Let $m\in \N_{> 0}$ be the minimal element in $\N_{>0}$ such that $x_1 < a_m < x_2$. Since $x_2 \notin J$, we have 
	\[[a_m,b_m)\subseteq (x_1,x_2) \subseteq I.\]
	Since $m$ was chosen to be minimal, we have that $a_n< x_1$ or $a_n>x_2$ for all $n\in \N_{>0}$ with $n<m$. Thus for all $n\in \N_{>0}$ with $n<m$ we have that $[a_n,b_n) \cap (x_1,x_2) = \emptyset$. Hence
	\[
	[a_m,b_m) \cap \bigcup_{n=0}^{m-1} [a_n,b_n) = \emptyset. 
	\]
	\end{proof}

\noindent For the rest of this section, \emph{fix $k,\ell \in \N_{>1}$ such that $k<\ell$ and $k,\ell$ are multiplicatively independent}. In the remainder of this section we prove the following lemma.

\begin{lem}\label{lem:mainlemma}
Let $R \in \N$. Then there exists $N \in \N_{>R}$ such that for all $n \in \N_{\geq N}$ and $r\in \{0,\dots,R-1\}$, there exist $s \in \N_{>n}$ and $t \in \N_{>0}$ such that
\begin{enumerate}
    \item $\ell^{-t}[k^{s}, k^{s} + k^R) \subseteq [k^n + k^r, k^n + k^{r+1})$,
    \item  for all $u\in \N$ and $v \in \{1,\dots,t-1\}$, 
    \[\ell^{-t}[k^{s}, k^{s} + k^R) \cap \ell^{-v}[k^{u}, k^{u} + k^R) = \emptyset.\]
\end{enumerate}
\end{lem}

\noindent As mentioned above, this lemma plays the central role in the proof of Theorem B. It allows us to find $\ell^{a},\ell^{b} \in \ell^{\N}$ for which the second-largest powers of $k$ appearing in the $k$-ary representations are small, while the second-largest powers of $k$ appearing in the $k$-ary representations of elements of $\ell^{\N}$ between $\ell^{a}$ and $\ell^{b}$ are relatively large. In order to motivate this lemma, we make this last statement more precise. 
Let $R,N,n,r,s,t\in \N$ satisfy the assumptions of Lemma \ref{lem:mainlemma} and its conclusion. By density of $k^{-\N}\ell^{\N}$ in $\R_{>0}$, there exist $i,j\in \N_{>0}$ such that
\begin{equation}\label{eq:explain}
k^{-i}\ell^j \in \ell^{-t}[k^s, k^s + k^R) \subseteq [k^n + k^r, k^n + k^{r+1}).
\end{equation}
Multiplying by $k^i$, we get that $\ell^j\in [k^{n+i} + k^{r+i}, k^{n+i} + k^{r+i+1})$. Since $r<n$, this implies that the two largest powers of $k$ appearing with non-zero coefficient in the $k$-ary representation of $\ell^j$ are $k^{n+i}$ and $k^{r+i}$. More is true; multiplying \eqref{eq:explain} by $k^i\ell^t$, we have that
$\ell^{j+t}\in [k^{s+i}, k^{s+i} + k^{R+i}]$. Thus the largest power appearing with non-zero coefficient in the $k$-ary representation of $\ell^{j+t}$ is $k^{s+i}$, and the second-largest power is smaller than or equal to $k^{R+i}$, as is the case for $\ell^j$. However, multiplying (2) by $k^i$, we get that for all $v \in \{1,\dots,t-1\}$, there is no $u\in \N$ such that
\[
\ell^{j+v} \notin [k^{u+i},k^{u+i}+k^{R+i}).
\]
Hence the second-largest power appearing with non-zero coefficient in the $k$-ary representation of $\ell^{j+v}$ must be larger than $k^{R+i}$.
\newline 

\noindent Before we start with the proof of Lemma \ref{lem:mainlemma} we fix some notation. We consider the group $([1,k),\odot,1)$, where
\[
x\odot y =  \begin{cases}
    x y , & \text{if } 1 \leq xy < k \\
    \frac{xy}{k}, & \text{otherwise}.
  \end{cases}
\]
Let $\lambda : \R_{>0} \to k^{\Z}$ map $x$ to $\max ((-\infty, x] \cap k^{\Z})$, and let $\mu: \mathbb{R}_{>0} \to [1,k)$ map $x$ to $\frac{x}{\lambda(x)}$. Observe that $\mu$ is a group homomorphism between $(\R_{>0},\cdot,1)$ and $([1,k),\odot,1)$ whose kernel is $k^{\Z}.$
 Let $\nu : \R_{>0} \to k^{\Z}$ map $x$ to the unique $k^n \in \Z$ such that $x \in \big[k^{n - \tfrac{1}{2}}, k^{n + \tfrac{1}{2}}\big)$. Note that $(\mu, \nu) : \R_{>0} \to [1, k) \times k^{\Z}$ is a bijection. See Figure \ref{fig:gridexample} for a visualization of these functions.
 
 \begin{figure}[t]
\begin{tikzpicture}[scale=0.6, every node/.style={scale=0.6}]

\draw (0,0) .. controls (0,1) and (6,2) .. (6,1);
\draw (6,1) .. controls (6,0) and (0,1) .. (0,2);
\draw (0,2) .. controls (0,3) and (6,4) .. (6,3);
\draw (6,3) .. controls (6,2) and (0,3) .. (0,4);
\draw (0,4) .. controls (0,5) and (6,6) .. (6,5);
\draw (6,5) .. controls (6,4) and (0,5) .. (0,6);
\draw (0,6) .. controls (0,7) and (6,8) .. (6,7);
\draw (6,7) .. controls (6,6) and (0,7) .. (0,8);

\filldraw (0,0) circle (3pt);
\node at (0,0.4) {$\ell^0$};
\filldraw (5.58,2.68) circle (3pt);
\node at (5.58,3.08) {$\ell^1$};
\filldraw (1.55,6.92) circle (3pt);
\node at (1.55,7.32) {$\ell^2$};

\draw (0,-3) .. controls (0,-2) and (6,-2) .. (6,-3);
\draw (0,-3) .. controls (0,-4) and (6,-4) .. (6,-3);

\filldraw (0,-3) circle (3pt);
\node at (0.6,-2.9) {$\mu(\ell^0)$};
\filldraw (5.58,-3.4) circle (3pt);
\node[below] at (5.58,-3.5) {$\mu(\ell^1)$};
\filldraw (1.55,-2.34) circle (3pt);
\node[right] at (1.55,-1.94) {$\mu(\ell^2)$};

\path[->, ultra thick] (3,0) edge node [right] {\scalebox{1.5}{$\boldsymbol{\mu}$}} (3,-2);

\draw[ultra thick] (6.5,1.2) .. controls (7,1.2) and (6.5,2) .. (7,2);
\draw[ultra thick] (6.5,2.8) .. controls (7,2.8) and (6.5,2) .. (7,2);
\draw[ultra thick] (6.5,3.2) .. controls (7,3.2) and (6.5,4) .. (7,4);
\draw[ultra thick] (6.5,4.8) .. controls (7,4.8) and (6.5,4) .. (7,4);
\draw[ultra thick] (6.5,5.2) .. controls (7,5.2) and (6.5,6) .. (7,6);
\draw[ultra thick] (6.5,6.8) .. controls (7,6.8) and (6.5,6) .. (7,6);

\path[->, ultra thick] (7.2,2) edge (9,2);
\path[->, ultra thick] (7.2,4) edge (9,4);
\path[->, ultra thick] (7.2,6) edge node [above] {\scalebox{1.5}{$\boldsymbol{\nu}$}} (9,6);

\filldraw (9.5, 2) circle (6pt);
\filldraw (9.5, 4) circle (6pt);
\filldraw (9.5, 6) circle (6pt);

\path [dashed] (9.5,0) edge (9.5,8);

\node at (-1.5,4) {\scalebox{2}{$\boldsymbol{\mathbb{R}_{>0}}$}};
\node at (-1.5,-3) {\scalebox{2}{$\boldsymbol{[1,k)}$}};
\node at (11,4) {\scalebox{2}{$\boldsymbol{k^\mathbb{Z}}$}};

\end{tikzpicture}
    \caption{Visualization of the functions $\mu$ and $\nu$ from the proof of Lemma \ref{lem:mainlemma}}
    \label{fig:gridexample}
\end{figure}
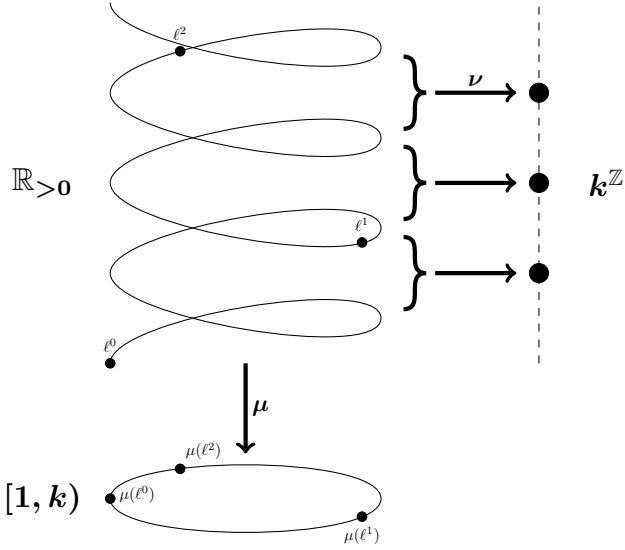

\subsection{Proof of Lemma \ref{lem:mainlemma}} 
Fix $R \in \N$ for the rest of this section. Let $E\in \N$ be sufficiently large such that $k^{-E} \leq \frac{1}{4}$. Obviously, $E>1$. Let $M$ be sufficiently large such that $\ell^M > k^{R+E}$. Take $\delta\in \R_{>1}$ sufficiently close to $1$ such that for all $i,j\in\{0,\dots,M\}$ with $i\neq j$
\[
 \mu\Big(\ell^{-i}(\delta^{-1},\delta)\Big)\cap \mu\Big(\ell^{-j}(\delta^{-1},\delta)\Big) = \emptyset.
\]
Note that such a $\delta$ exists, because $k$ and $\ell$ are multiplicatively independent. Take $N$ to be sufficiently large such that $\frac{k^N + k^R}{k^N} < \delta$ and $N > R + E$. We will now show that $N$ satisfies the conclusion of Lemma \ref{lem:mainlemma}.\newline

\noindent Let $n\in \N_{\geq N}$ and $r\in \{0,\dots,R-1\}$. We now find $s\in \N_{>n}$ and $t\in \N_{>0}$ such that statements (1) and (2) of Lemma \ref{lem:mainlemma} hold.  \newline

\noindent Set $I_0=[k^n + k^r, k^n + k^{r+1})$.
Since  $n \geq N > R + E$, we have that 
\begin{equation}\label{eq:prelemma} 
k^n + k^{r+1} \leq k^n + k^{R} \leq k^n (1 + k^{-E}) \leq k^n \cdot \frac{5}{4} < k^{n+\tfrac{1}{2}}.
\end{equation}
Thus $\nu(I_0) = \{ k^n \}.$

\begin{lem}\label{lem:unique_u}
Let $u\in\N_{>0}$ and $v \in \N_{>0}$ be such that $\ell^{-v}[k^{u}, k^{u} + k^n)\cap I_0\neq \emptyset$. Then $u\geq n$ and 
\[
u = \log_k\left(\frac{k^n}{\nu(\ell^{-v})}\right).
\]
\end{lem}
\begin{proof}
We first show that $u\geq n$. Towards a contradiction, suppose that $u<n$. Since $v>0$ and $\ell > k$, we have that $k^u\ell^{-v}<k^{n-2}$. Since $R<n-1$, we get
\[
\ell^{-v}(k^u + k^R) < 2 k^{n-2} \leq k^{n-1}.
\]
Hence $k^n \notin \nu(\ell^{-v}[k^{u}, k^{u} + k^R))$. This contradicts our previous observation that $\nu(I_0)=\{k^n\}$ by \eqref{eq:prelemma}.
\newline

\noindent Let $a,b \in \R$ be such that $[a,b)\cap I_0\neq \emptyset$ and $\frac{b}{a} \leq  \frac{k^n + k^R}{k^n}$. We now show that $\nu(a)=k^n$. Since $b>k^n$,
\[
a \geq k^n \frac{k^n}{k^n + k^R} > k^n \frac{k^n}{k^{n+\tfrac{1}{2}} } = k^{n - \tfrac{1}{2}}.
\]
Since $a<k^n + k^{r+1}$, we have $\nu(a)=k^n$.\newline

\noindent As $u \geq n$, we have that $\frac{k^{u} + k^R}{k^{u}} \leq \frac{k^n + k^R}{k^n}$. Thus by the previous paragraph  $\nu(\ell^{-v} k^{u}) = k^n$ and  $k^{u}=\frac{k^n}{\nu(\ell^{-v})}.$
\end{proof}

\noindent We now introduce a function $U: \N_{>0} \to \N_{>0}$ mapping $v\in \N$ to $\log_k(\frac{k^n}{\nu(\ell^{-v})})$. By Lemma \ref{lem:unique_u}, for all $u\in\N_{>0}$ and $v \in \N_{>0}$, we have that $u=U(v)$ whenever $\ell^{-v}[k^{u}, k^{u} + k^R)\cap I_0\neq \emptyset$.\newline

\noindent For $v\in \N_{>0}$, we let $I_v$ be $ \ell^{-v}[k^{U(v)}, k^{U(v)} + k^R)$. The length of $I_v$ is $\ell^{-v} k^R$.

\begin{lem}\label{lem:intervalsize}
Let $v\in \N$. Then $\mu(I_{v})\subseteq \mu(\ell^{-v})\odot [1,\delta)$.
\end{lem}
\begin{proof}
Suppose that $v\in \N_{>0}$ and let $x\in \mu(I_{v})$. Since $\mu$ is a homomorphism and $\mu(k^{\Z})=1$, 
 \[
 x \in \mu\Big( k^{-U(v)}\ell^{-v} \big[k^{U(v)}, k^{U(v)} + k^R\big)\Big).
 \]
Because $ \frac{k^{U(v)} + k^R}{k^{U(v)}} < \frac{k^n + k^R}{k^n} \leq \frac{k^N + k^R}{k^N} < \delta$ and $\mu$ is a homomorphism, we have
 \[
 x \in \mu\Big(\ell^{-v}\big[1,\delta\big)\Big)=\mu(\ell^{-v})\odot [1,\delta).
 \]
 Since $I_0 \subseteq [k^n,k^n+k^R)$, the same argument gives that $\mu(I_0) \subseteq [1,\delta)$.
\end{proof}
		
\noindent  
For $v\in \N$, set $K_v:= \mu\big(\ell^{-v} (\delta^{-1},\delta)\big)$. By Lemma \ref{lem:intervalsize}, $\mu(I_v)\subseteq K_v$.

\begin{table}[t]
    \centering
\begin{tabular}{ c|c }
Name & Function value at $x$\\
 \hline

 $\lambda: \R_{>0} \to k^{\Z}$ &  $\max \{ k^n : k^n\leq x\}$\\
$\mu : \R_{>0} \to [1,k)$&
 $\frac{x}{\lambda(x)}$ \\
 $\nu: \R_{>0} \to k^{\Z}$ & $\max \{ k^n \ : k^{n - \tfrac{1}{2}}\leq x < k^{n + \tfrac{1}{2}}\}$\\
\end{tabular}
    \caption{Definitions of functions used in the proof of Lemma \ref{lem:mainlemma}}
    \label{fig:my_label}
\end{table}

\begin{lem}\label{lem:disjoint_k}
Let $v\in \N$. Then $K_v, K_{v+1}, \dots, K_{v+M}$ are disjoint.
\end{lem}
\begin{proof}
 Let $i,j\in \{0,\dots,M\}$ be such that $i\neq j$. Towards a contradiction, suppose there is $x \in K_{v+i}\cap K_{v+j}$. Since $\mu$ is a homomorphism, we get 
\[
x\in \mu\big(\ell^{-v-i} (\delta^{-1},\delta)\big) \cap  \mu\big(\ell^{-v-j} (\delta^{-1},\delta)\big)\big) \subseteq \mu(\ell^{-v})\odot \Big( \mu\big(\ell^{-i}(\delta^{-1},\delta)\big) \cap  \mu\big(\ell^{-j}(\delta^{-1},\delta)\big)\Big)
\]
This is a contradiction, because by our choice of $\delta$ the set $\mu\big(\ell^{-i}(\delta^{-1},\delta)\big) \cap  \mu\big(\ell^{-j}(\delta^{-1},\delta)\big)$ is empty.
\end{proof}

\begin{lem}\label{lem:intersecting_i}
Let $v_1,v_2 \in \N$ be such that $v_1<v_2$ and $I_{v_1}\cap I_0\neq \emptyset$ and $I_{v_2}\cap I_0\neq \emptyset$. Then $v_2 > v_1 + M$.
\end{lem}
\begin{proof}
Towards a contradiction, suppose that $v_2 = v_1 +i$, where $i \in \{1,\dots,M\}$. Let $x\in \mu(I_{v_1}\cap I_0)$ and $y\in \mu(I_{v_2}\cap I_0)$. Since $\mu(I_0) \subseteq [1,\delta)$, we have $\frac{x}{y} \in (\delta^{-1},\delta)$. We only consider the case that $\frac{x}{y}  \in (\delta^{-1},1]$, as the other case can be handled in the same way with the roles of $x$ and $y$ reversed. Let $c\in (\delta^{-1},1]$ be such that $x=cy$. Since $x=\mu(x)$, we have
\[
x=cy=\mu(cy)=\mu(c) \odot \mu(y) \in  \mu\big((\delta^{-1},1]\big)\odot \mu(I_{v_2}).
\]
Thus by Lemma \ref{lem:intervalsize}
\[
x \in \mu\big((\delta^{-1},1]) \odot \mu(I_{v_2}) \subseteq  \mu(\ell^{-v_2})\odot \mu\big((\delta^{-1},\delta)\big)= K_{v_2}.
\]
Since $\mu(I_{v_1})\subseteq K_{v_1}$, we have $x\in K_{v_1}\cap K_{v_2}$. This contradicts Lemma \ref{lem:disjoint_k}.
\end{proof}

\noindent We now finish the proof of Lemma \ref{lem:mainlemma}. Let $(m_i)_{i > 0}$ be an increasing sequence of elements of $\N_{>0}$ containing all $m_i > 0$ such that $I_{m_i} \cap I_0$ is nonempty. By Lemma \ref{lem:intersecting_i}, we have that $m_1 > M$, and $m_{i+1} > M + m_i$ for all $i\in \N_{>0}$. We will show that the interval $I_0$ and the sequence $(I_{m_i})_{i>0}$ satisfy the assumption of Lemma \ref{lem:intervals}.\newline

\noindent Let $L_i$ be the length of $I_{m_i}$, and let $L_0$ be the length of $I_0$. We get 
\[
\frac{L_1}{L_0} = \frac{\ell^{-m_1} k^R}{k^{r+1}-k^r} < \ell^{-m_1} k^R < \ell^{-M} k^R < k^{-R-E} k^R \leq \frac{1}{4}.
\]
Similarly, for $i \geq 1$ we obtain
\[
\frac{L_{i+1}}{L_i} = \frac{\ell^{-m_{i+1}} k^R}{\ell^{-m_i} k^R} = \ell^{m_i - m_{i+1}} < \ell^{-M} < k^{-R-E} < k^{-E} \leq \frac{1}{4}.
\]

\noindent Next, we prove that that the left endpoints of the $I_{m_i}$'s are dense in $I_0$.
Let $J\subseteq I_0$ be an open subinterval of $I_0$. By density of $\ell^{-\N}k^{\N}$ in $\R_{>0}$, there is $m\in \N$ and $n\in \N$ such that $\ell^{-m}k^{n}$ is in $J$.  Thus $n=U(m)$ and there is $i\in \N$ such that $m_i=m$. Hence $\ell^{-m}k^{n}$ is the left endpoint of $I_{m_i}$.\newline

\noindent By Lemma \ref{lem:intervals} there is $j \in \N_{>0}$ such that $I_{m_j} \subseteq I_0$ and $I_{m_j}
\cap \bigcup_{i=1}^{j-1} I_{m_i}=\emptyset.$ Set $t:=m_j$ and $s:=U(t)$. We will now show that the conclusion of Lemma \ref{lem:mainlemma} holds for this choice of $t$ and $s$. Note that statement (1) is immediate from the choice of $j$. So we just need to show statement (2).\newline

\noindent Towards a contradiction, suppose there are  $u\in \N_{>0}$ and $v\in \{1,\dots,t-1\}$ such that
\[
\ell^{-t}[k^{s}, k^{s} + k^R) \cap \ell^{-v}[k^{u}, k^{u} + k^R) \neq \emptyset.
\]
Since $\ell^{-t}[k^{s}, k^{s} + k^R)\subseteq I_0$, we have that
\[
\ell^{-v}[k^{u}, k^{u} + k^R) \cap I_0 \neq \emptyset.
\]
Thus by Lemma \ref{lem:unique_u}, $u=U(v)$. So there is $j' \in \{0,\dots,j\}$ such that 
\[
\ell^{-v}[k^{u}, k^{u} + k^R)= I_{m_{j'}}.
\]
Then $\ell^{-t}[k^{s}, k^{s} + k^R) \cap \ell^{-v}[k^{u}, k^{u} + k^R) = \emptyset$, contradicting our assumption. \qed

\section{Undecidability of $(\N,+,k^{\N},\ell^{\N})$}

The goal of this section is to use Lemma \ref{lem:mainlemma} to show that the theory of $(\N,+,k^{\N},\ell^{\N})$ is sufficiently powerful to encode statements about Turing machines.\newline

\noindent We recall that the function $\lambda : \N_{>0} \to k^{\Z}$ maps $y$ to $\max ((-\infty, y] \cap k^{\N})$ and  observe that $\lambda$ is definable in $(\N,+,k^{\N})$. Indeed, the formula expressing $\lambda(y)=z$ in the language of $(\N,+,k^{\N})$ is
\[
z \in k^{\N} \wedge z \leq y \wedge \forall u \in k^{\N} (u\leq y \rightarrow u\leq z).
\]
It is worth keeping in mind that for every $y\in \N_{>0}$ the number $\lambda(y)$ is the largest power of $k$ that appears with a non-zero digit in the $k$-ary representation of $y.$ For $i\in \N$, we denote by $\sigma_1^i$ the $i$-th successor function on $k^{\N}$; that is $\sigma_1^i$ is the function that maps $k^n$ to $k^{n+i}$ for all $n\in \N$.  
\newline 

\begin{defn}  
We define 
\[
S := \{ (x,y) \in k^{\N} \times \ell^{\N} \ : \ \lambda(y-\lambda(y)) =x  \wedge x < \lambda(y)\}.
\]
For $x \in k^{\N}$, we set 
\[
S(x):=\{ y \in \ell^{\N} \ : \ (x,y) \in S\}.
\]
\end{defn}
\noindent It is clear that $S$ is definable in $(\N,+,k^{\N},\ell^{\N})$, since $\lambda$ is definable in $(\N,+,k^{\N})$.

\begin{ex}
Let $k=2$ and $\ell=3$. Consider the following binary expansions of powers of $3$:
\begin{align*}
    3^5 &= 243  = 128+64+32+16+2+1=2^7 + 2^6 + 2^5 + 2^4 + 2^1 + 2^0,\\
    3^6 &= 729 = 512 + 128+ 64 + 16 + 8 + 1= 2^9 + 2^7 + 2^6+2^4+ 2^3+2^0,\\
    3^7 &= 2187 = 2048 + 128 + 8 + 2 + 1 = 2^{11} + 2^7 + 2^3+ 2^1+2^0.
\end{align*}
Thus $3^5\in S(2^6)$ and $3^6,3^7 \in S(2^7)$.
\end{ex}

\begin{lem}\label{lem:sdef}
Let $(x,y) \in k^{\N} \times \ell^{\N}$. Then the following are equivalent:
\begin{enumerate}
    \item $(x,y) \in S,$ 
    \item there is $z\in k^{\N}$ such that $x<z$ and $y \in [z+x,z+kx),$
    \item $1$ is the most significant digit in the $k$-ary representation of $y$, and $x$ is the second largest power of $k$ that appears with a non-zero digit in the $k$-ary representation of $y$.
\end{enumerate}
\end{lem}
\begin{proof}
The equivalence of (2) and (3) follows easily from the greediness of  $k$-ary representations.\newline

\noindent Suppose (1) holds; that is $(x,y)\in S$. Then set $z:=\lambda(y)$. By the definition of $S$ and $\lambda$, we get that $x<z$ and $x \leq y-z < kx$. Thus (2) holds.\newline

\noindent Suppose (2) holds. Let $z\in k^n$ be such that $x < z$ and $z+x\leq y < z+kx$. Thus $x \leq y-z < kx$, and hence $\lambda(y-z) = x$. In order to show that $(x,y)\in S$, it is left to show that $\lambda(y)=z$. Since $x<z$, we know that $z \leq y < z+kz$. Therefore $\lambda(y)=z$. 
\end{proof}

\noindent We now establish that we can find every pattern in tuples of subsets of $\ell^{\N}$ of the form
\[
(S(k^n),S(k^{n+1}),\dots,S(k^{n+s})).
\]
The next lemma makes this statement precise and the following example illustrates it.

\begin{lem}\label{lem:pickoutvalues}
	Let $R\in \N$, and let $r=r_0\cdots r_s$ be an $\{0,1,\dots,R\}$-word. Then there are $n \in \N_{>0}$ and an increasing sequence $m_0<m_1<\dots < m_s$ of natural numbers such that for all $i\in\{0,\dots,R-1\}$
\[
S(k^{n + i})\cap \big[\ell^{m_0},\ell^{m_s}\big] =\{ \ell^{m_j} \ : \  j \in \{0,\dots, s\} \wedge r_j = i\}.
\]
\end{lem}

\noindent Before we give the proof of Lemma \ref{lem:pickoutvalues}, let us consider an example application of it.

\begin{ex}\label{ex:re}
Let $R=3$ and let $r=012011023330$. Since the word $r$ has length $12$, we have that $s=11$. Then Lemma \ref{lem:pickoutvalues} states that there is $n\in \N$ and natural numbers $m_0<m_1<\dots < m_{11}$ such that
\begin{align*}
S(k^n)  \cap \big[\ell^{m_0},\ell^{m_{11}}\big] &= \{ \ell^{m_0}, \ell^{m_3}, \ell^{m_6}, \ell^{m_{11}}\},\\
S(k^{n+1})  \cap \big[\ell^{m_0},\ell^{m_{11}}\big] &= \{ \ell^{m_1}, \ell^{m_4}, \ell^{m_5}\},\\
S(k^{n+2})  \cap \big[\ell^{m_0},\ell^{m_{11}}\big] &= \{ \ell^{m_2}, \ell^{m_7}\},\\
S(k^{n+3})  \cap \big[\ell^{m_0},\ell^{m_{11}}\big] &= \{ \ell^{m_8}, \ell^{m_9}, \ell^{m_{10}}\}.
\end{align*}
The graphical representation of the above equations is in Figure \ref{fig:patternexample}. We can think of the word $r$ as instruction which row to mark with an X in a given column. Thus Lemma  \ref{lem:pickoutvalues} tells us that there exist $n$ and $m_0,\dots,m_{11}$ such that the tuple
\[
\Big(S(k^n)  \cap \big[\ell^{m_0},\ell^{m_{11}}\big] ,
S(k^{n+1})  \cap \big[\ell^{m_0},\ell^{m_{11}}\big] ,
S(k^{n+2})  \cap \big[\ell^{m_0},\ell^{m_{11}}\big] ,
S(k^{n+3})  \cap \big[\ell^{m_0},\ell^{m_{11}}\big]\Big)
\]
realizes this pattern.
\end{ex}

\begin{figure}[t]
    \centering
   \begin{tikzpicture}[scale=0.7]
\node at (1.5,4.5) {$\ell^{m_0}$};
\node at (1.5,3.5) {X};
\node at (2.5,4.5) {$\ell^{m_1}$};
\node at (3.5,4.5) {$\ell^{m_2}$};
\node at (4.5,4.5) {$\ell^{m_3}$};
\node at (5.5,4.5) {$\ell^{m_4}$};
\node at (6.5,4.5) {$\ell^{m_5}$};
\node at (7.5,4.5) {$\ell^{m_6}$};
\node at (8.5,4.5) {$\ell^{m_7}$};
\node at (9.5,4.5) {$\ell^{m_8}$};
\node at (10.5,4.5) {$\ell^{m_{9}}$};
\node at (11.5,4.5) {$\ell^{m_{10}}$};
\node at (12.5,4.5) {$\ell^{m_{11}}$};
\node at (0.4,3.5) {$k^{n}$};
\node at (0.4,2.5) {$k^{n+1}$};
\node at (0.4,1.5) {$k^{n+2}$};
\node at (0.4,0.5) {$k^{n+3}$};

 \foreach \y in {1,2,3,4}{
    \foreach \x in {1,...,12}{
        \draw (\x,\y-1) rectangle (1+\x,\y);}}
\node at (1.5,3.5) {X};
\node at (4.5,3.5) {X};
\node at (7.5,3.5) {X};
\node at (12.5,3.5) {X};
\node at (2.5,2.5) {X};
\node at (5.5,2.5) {X};
\node at (3.5,1.5) {X};
\node at (8.5,1.5) {X};
\node at (6.5,2.5) {X};
\node at (10.5,0.5) {X};
\node at (11.5,0.5) {X};
\node at (9.5,0.5) {X};
\end{tikzpicture}
    \caption{An example of a pattern}
    \label{fig:patternexample}
\end{figure}

\begin{proof}[Proof of Lemma \ref{lem:pickoutvalues}]
Let $N$ be defined as in Lemma \ref{lem:mainlemma} (note that this definition only depends on $R$). Let $n_0 = N$. Then by Lemma \ref{lem:mainlemma}, there are $n_1 \in \N_{>n_0}$ and $M_1 \in \N_{>0}$ such that 
\begin{itemize}
    \item $\ell^{-M_1} [k^{n_1}, k^{n_1} + k^R) \subseteq [k^{n_0} + k^{r_0}, k^{n_0} + k^{r_0 + 1})$,
    \item for all $u\in \N$ and $v \in \{0,\dots,M_1-1\}$, 
\[
\ell^{-M_{1}} [k^{n_{1}}, k^{n_{1}} + k^R) \cap \ell^{-v} [k^{u}, k^{u} + k^R) = \emptyset.
\]
\end{itemize}
 We repeat this process to produce $n_2, \dots, n_s\in \N$ and $M_2, \dots, M_s\in \N_{>0}$ such that $n_0<n_1<n_2 < \dots < n_s$ and
\begin{align}\label{eq:inclusions}
\begin{split}
[k^{n_0} + k^{r_0}&,  k^{n_0} + k^{r_0 + 1})\\
\supseteq \ell^{-M_1} [k^{n_1}, k^{n_1} + k^R) &\supseteq \ell^{-M_1} [k^{n_1} + k^{r_1}, k^{n_1} + k^{r_1 + 1})\\
		\supseteq \ell^{-M_1-M_2} [k^{n_2}, k^{n_2} + k^R) &\supseteq \ell^{-M_1-M_2} [k^{n_2} + k^{r_2}, k^{n_2} + k^{r_2 + 1})\\
		&\vdots\\
		\supseteq \ell^{-M_1-M_2-\dots-M_s} [k^{n_s}, k^{n_s} + k^R) &\supseteq \ell^{-M_1-M_2-\dots-M_s} [k^{n_s} + k^{r_s}, k^{n_s} + k^{r_s + 1}),
\end{split}
\end{align}	
and for all $j\in\{1,\dots,s\}$, $u\in \N$, and $v \in \{0,\dots,M_j-1\}$, 
\begin{equation}\label{eq:empty}
    \ell^{-M_{j}} [k^{n_{j}}, k^{n_{j}} + k^R) \cap \ell^{-v} [k^{u}, k^{u} + k^R) = \emptyset.
\end{equation}
By density of $k^{-\N}\ell^{\N}$ in $\R_{>0}$, there are $n, M \in \N_{> 0}$ such that 
\[
k^{-n} \ell^M\in \ell^{-M_1-M_2-\dots-M_s} [k^{n_s} + k^{r_s}, k^{n_s} + k^{r_s + 1}).
\]
Unrolling \eqref{eq:inclusions}, we obtain
		\begin{align}\label{eq:s}
		\begin{split}
		\ell^{M+M_1+M_2+\dots+M_s} \in& \ [k^{n + n_s} + k^{n + r_s}, k^{n + n_s} + k^{n + r_s + 1}),\\
		\vdots\\
		\ell^{M+M_1+M_2} \in& \ [k^{n + n_2} + k^{n + r_2}, k^{n + n_2} + k^{n + r_2 + 1}),\\
		\ell^{M+M_1} \in& \ [k^{n + n_1} + k^{n + r_1}, k^{n + n_1} + k^{n + r_1 + 1}),\\
		\ell^{M} \in& \ [k^{n + n_0} + k^{n + r_0}, k^{n + n_0} + k^{n + r_0 + 1}).
		\end{split}
		\end{align}
		Set $m_0:=M$, and for $j>0$, set $m_j := M + M_1 + \dots + M_j$. By the above and Lemma \ref{lem:sdef}, we have that $\ell^{m_j} \in S(k^{n + r_j})$ for $j=0,\dots,s$.\newline
		
		\noindent Towards a contradiction, suppose there are some $m\in \N$, $i\in\{0,\dots,R-1\}$ and $j \in \{0,\dots,s-1\}$ such that $m_j < m < m_{j+1}$ and $\ell^m \in S(k^{n + i})$. Since $\ell^m \in S(k^{n + i})$, there exists $n' \in \N$ with $n'> n + i$ such that
		\[
		\ell^m \in [k^{n'} + k^{n + i}, k^{n'} + k^{n + i + 1}).
		\]
		It follows that $\ell^m \in [k^{n'}, k^{n'} + k^{n + R})$. Multiplying each side by $k^{-n}\ell^{-(m - m_j)}$, we obtain 
		\[
		k^{-n} \ell^{m_j} \in \ell^{-(m - m_j)} [k^{n'-n}, k^{n'-n} + k^R).
		\]
By \eqref{eq:s}, $$\ell^{m_{j+1}} \in [k^{n + n_{j+1}} + k^{n + r_{j+1}}, k^{n + n_{j+1}} + k^{n + r_{j+1} + 1})\subseteq [k^{n + n_{j+1}}, k^{n + n_{j+1}} + k^{n + R}).$$
Multiplying both sides by $k^{-n}\ell^{-M_{j+1}}$ and using that $m_{j+1}=m_j + M_{j+1}$, we get		
\[
k^{-n} \ell^{m_j} \in \ell^{-M_{j+1}} [k^{n_{j+1}}, k^{n_{j+1}} + k^R).
\]
So we have $(m - m_j) \in \{1,\dots,M_{j+1}\}$ and an integer $(n' - n)\in \N_{>0}$ such that
	$$\ell^{-M_{j+1}} [k^{n_{j+1}}, k^{n_{j+1}} + k^R) \cap \ell^{-(m - m_j)} [k^{n'-n}, k^{n'-n} + k^R) \neq \emptyset.$$
This contradicts \eqref{eq:empty}. Thus we conclude that no such $m$ exists.
\end{proof}

\subsection{Coding finite subsets of $\N^2$}
Let $u, R\in \N$. In this subsection, we introduce a way to code $u$-tuples of subsets of $\{0,1,\dots,R-1\}^2$ as finite words over the alphabet $\{0,1,\dots,R\}$. In order to so, we first code such tuples as $R\times R$-matrices.\newline

\noindent For $X\subseteq \N^2$, let $\chi_X$ denote the \textbf{characteristic function} of $X$; that is the function that maps $x\in \N^2$ to $1$ if $x\in X$, and to $0$ otherwise.

\begin{defn}Let $X=(X_1,\dots, X_u)\in \mathcal{P}(\N^2)^u$ and let $R\in \N$. Define $B(X,R)$ to be the $R\times R$-matrix $(B(X,R)_{i,j})_{1\leq i,j\leq R}$ such that for all $i,j \in \{1,\dots,R\}$
\[
B(X,R)_{i,j} = \sum_{v=0}^{u-1} \chi_{X_{v+1}}(i-1,j-1) 2^{v}.
\]
\end{defn}

\begin{ex}
Let $X_1 := \{(0,0),(1,0),(1,2),(2,2)\}$ and let $X_2 :=  \{(0,1),(2,2)\}$, and let $R=3$. Then
\[
B((X_1,X_2),3) = \begin{bmatrix} 1 & 2 & 0\\  1 & 0 & 1\\ 0 & 0 & 3\end{bmatrix}.
\]
Elements of $X_1$ or $X_2$ one of whose coordinates is at least $R$, do not affect the computation of this matrix. For example, if $X_1':=X_1 \cup \{(3,0)\}$, then we still have  
\[
B((X_1',X_2),3) = \begin{bmatrix} 1 & 2 & 0\\  1 & 0 & 1\\ 0 & 0 & 3\end{bmatrix}.
\]
\end{ex}

\noindent It is clear that this definition allows us to encode $u$-tuples of subsets of $\{0,1,\dots,R-1\}^2$ as $R\times R$-matrices with entries in $\{0,\dots,2^{u}-1\}$. We leave the routine proof of the following lemma to the reader. 

\begin{lem}\label{lem:TtoM} Let $X=(X_1,\dots, X_u)\in \mathcal{P}(\N^2)^u$ and let $R,i,j\in \N$ be such that $0 \leq i,j< R$. Then for all $v\in \{1,\dots,u\}$ the following are equivalent:
\begin{enumerate}
    \item $(i,j) \in X_v$,
    \item $2^{v-1}$ appears in the binary expansion of $B(X,R)_{i+1,j+1}$.
\end{enumerate}
\end{lem}

\noindent As the next step, we encode $R\times R$-matrices with entries in $\N$ as a finite $\{0,1\dots,R\}$-word.

\begin{defn} Let $B=(B_{ij})_{1\leq i,j\leq R}$ be an $R\times R$-matrix with entries in $\N$. We now define a $\{0,1,\dots,R\}$-word $w(B)$ as follows:
for $j=1,\dots,R$
\begin{enumerate}
    \item set $v_{0j} := 0,$
    \item for $i=1,\dots,R$, set $
    v_{ij} := \underbrace{i\cdots i}_{B_{ij}\text{-times}},$
    \item set $v_j := v_{0j}v_{1j}\cdots v_{Rj}$.
\end{enumerate}
Set $w(B):=v_1v_2\cdots v_R0$.
\end{defn}
\noindent Another way of writing this is
\[
w(B) = 0 1^{B_{11}}2^{B_{21}}\cdots R^{B_{R1}}0\cdots 01^{B_{1R}}2^{B_{2R}}\cdots R^{B_{RR}}0.
\]

\begin{ex}\label{ex:BwB}
Let $B=\begin{bmatrix} 1 & 2 & 0\\  1 & 0 & 1\\ 0 & 0 & 3\end{bmatrix}$. Observe that $v_{01}=0,v_{11}=1,v_{21}=2$ and $v_{31}$ is the empty word. Hence $v_1=012$. Furthermore, $v_{02}=0$, $v_{12}=11$ and both $v_{22}$ and $v_{23}$ are the empty word. Thus $v_2=011$. Finally, $v_{03}=0$, $v_{23}=2$, $v_{33}=333$, while $v_{13}$ is the empty word. Hence $v_3=02333$. Putting this together, we obtain
\[
w(B) = \underbrace{012}_{v_1}\underbrace{011}_{v_2}\underbrace{02333}_{v_3}0.
\]
\end{ex}

\noindent Note that $0$ appears $(R+1)$-times in $w(B)$, including at the beginning and at the end of $w(B)$. This allows us to recover the matrix $B$ from the word $w(B)$ as follows.\newline

\noindent Let $r=r_0\cdots r_s$ be a finite word over $\N$. Let $Z(r)$ be the set
\[
\{ p \in \{0,\dots, s\} \ : \ r_p = 0 \}.
\]
For $j\in \N_{>0}$, we define $Z(r,j)$ to be the $j$-th element of $Z(r)$ with respect to $<$ if such exists, and $s$ otherwise. For $i,j\in \N_{>0}$, we define
\[
C(r,i,j) = \card \{ p\in \{0,\dots,s\} \ : \ Z(r,j) < p <Z(r,j+1), \ r_{p} = i \}.
\]
\begin{ex}
Let $r=012011023330$. Then $Z(r,1)=0$, $Z(r,2)=3$, $Z(r,3)=6$ and $Z(r,4)=11$. Computing $C(r,-,-)$, we obtain
\begin{align*}
C(r,1,1)&=1, \qquad C(r,1,2) =2, \qquad C(r,1,3) = 0,\\
C(r,2,1)&=1, \qquad C(r,2,2) =0, \qquad C(r,2,3) = 1,\\
C(r,3,1)&=0, \qquad C(r,3,2) =0, \qquad C(r,3,3) = 3.
\end{align*}
Let $B$ be the $3\times 3$-matrix in Example \ref{ex:BwB}. We observe that $r=w(B)$ and $B_{i,j}=C(r,i,j)$ for all $i,j\in \{1,2,3\}$.
\end{ex}

\noindent This last observation is always true.

\begin{lem}\label{lem:MtoW} Let $B=(B_{i,j})_{1\leq i,j\leq R}$ be an $R\times R$-matrix with entries in $\N$. Then for all $i,j\in \{1,\dots,R\}$ 
        \[
        B_{i,j} = C(w(B),i,j).
        \]
\end{lem}

\noindent We leave the straightforward proof of the preceding lemma to the reader. Now combining Lemma \ref{lem:TtoM} and Lemma \ref{lem:MtoW}, we obtain the following corollary. 

\begin{cor}\label{cor:count} Let $X=(X_1,\dots, X_u)\in \mathcal{P}(\N^2)^u$ and let $R,i,j\in \N$ be such that $0\leq i,j < R$. Then for all $v\in\{1,\dots,u\}$ the following are equivalent:
\begin{enumerate}
    \item $(i,j) \in X_v$,
    \item $2^{v-1}$ appears in the binary representation of $C\big(w(B(X,R)),i+1,j+1\big)$.
\end{enumerate}
\end{cor}

\subsection{Coding the halting problem} 
In this subsection we finally give the proof of Theorem B. The goal is to encode the halting problem in $\textrm{FO}(\N,+,k^{\N},\ell^{\N})$. In order to do so, we introduce first several sets definable in $(\N,+,k^{\N},\ell^{\N})$.
    
\begin{defn}
Let $D$ denote the set
    \[
    \{ (K_1,K_2,L_1,L_2) \in (k^{\N})^2 \times (\ell^{\N})^2 \ : \ K_1 \leq K_2 \wedge L_1 \leq L_2 \wedge L_1,L_2 \in S(K_1)\}.
    \]
For each $R_1,R_2\in \N$, we define $D(R_1,R_2)$ to be the set of all $(K_1,K_2,L_1,L_2)\in D$ such that
\[
 \card \big(k^{\N} \cap (K_1,K_2]\big) = R_1, \ \card \big(S(K_1) \cap [L_1,L_2)) = R_2.
\]
\end{defn}

\noindent Observe that $D$ is definable in $(\N,+,k^{\N},\ell^{\N})$. It is also true and that for each $R_1,R_2\in \N$, the set $D(R_1,R_2)$ is also definable in $(\N,+,k^{\N},\ell^{\N})$, but we are not going to use this fact. Also note that for $(K_1,K_2,L_1,L_2)\in D$, the successor function $\big(S(K_1) \cap [L_1,L_2))$ is definable in $(\N,+,k^{\N},\ell^{\N}).$
\begin{defn}
 Let $K \in k^{\N}$ and $L \in S(K)$.
Let $\sigma_2(K,L)$ be the successor of $L$ in $S(K)$ if such exists, and $0$ otherwise. We set $\sigma_2^0(K,L):=L$ and for $i\in \N$, we define $\sigma_2^{i+1}(K,L)$ to be $\sigma_2(K,\sigma_2^i(K,L)).$
\end{defn}

\noindent From the definitions, we immediately obtain the following lemma.

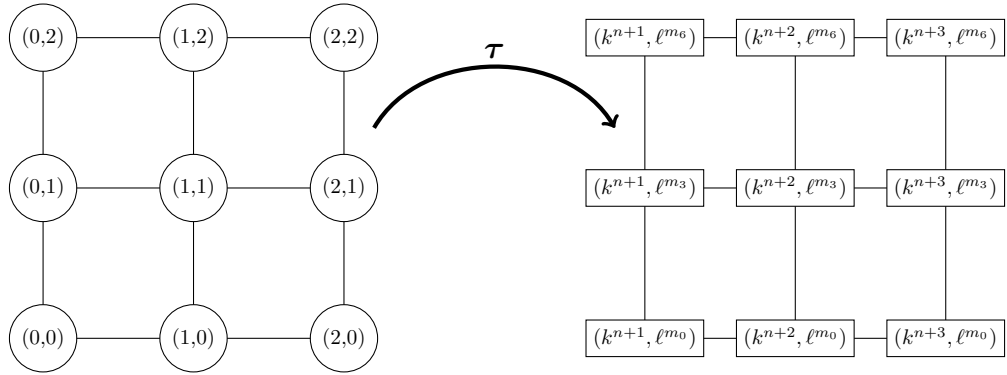
\begin{figure}[b]
\begin{tikzpicture}[scale=0.8, every node/.style={scale=0.8}]
\draw[step=2.5cm] (-0.2,-0.2) grid (5.2,5.2);
\foreach \i in {0,...,2}
{
\foreach \j in {0,...,2}
{
\draw node[draw,circle,fill=white] at (2.5*\i ,2.5*\j) {(\i,\j)};
}
}

\path[->, ultra thick] (5.5,3.5) edge[bend left=60] node [above] {\scalebox{1.5}{$\boldsymbol{\tau}$}} (9.5,3.5);

\def\x{10}

\draw[step=2.5cm] (-0.2+\x,-0.2) grid (5.2+\x,5.2);
\draw node[draw,rectangle,fill=white] at (0+\x,0) {($k^{n+1},\ell^{m_0})$};
\draw node[draw,rectangle,fill=white] at (2.5+\x,0) {($k^{n+2},\ell^{m_0})$};
\draw node[draw,rectangle,fill=white] at (5+\x,0) {($k^{n+3},\ell^{m_0})$};
\draw node[draw,rectangle,fill=white] at (0+\x,2.5) {($k^{n+1},\ell^{m_3})$};
\draw node[draw,rectangle,fill=white] at (2.5+\x,2.5) {($k^{n+2},\ell^{m_3})$};
\draw node[draw,rectangle,fill=white] at (5+\x,2.5) {($k^{n+3},\ell^{m_3})$};
\draw node[draw,rectangle,fill=white] at (0+\x,5) {($k^{n+1},\ell^{m_6})$};
\draw node[draw,rectangle,fill=white] at (2.5+\x,5) {($k^{n+2},\ell^{m_6})$};
\draw node[draw,rectangle,fill=white] at (5+\x,5) {($k^{n+3},\ell^{m_6})$};
\end{tikzpicture}
    \caption{Visualization of the function $\tau$ from Lemma \ref{lem:existsubset} in the case of Example \ref{ex:D}}
    \label{fig:gridexample}
\end{figure}

\begin{lem}\label{lem:existsubset} Let $R_1,R_2 \in \N$ and let $(K_1,K_2,L_1,L_2)\in D(R_1,R_2)$.
Then the function
\begin{align*}
 \tau : \{0,\dots, R_1-1\} \times \{0,\dots,R_2-1\}  &\to \big(k^{\N} \cap (K_1,K_2] \big)\times \big(S(K_1) \cap [L_1,L_2)\big)\\
    (i,j) &\mapsto (\sigma_1^{i+1}(K_1),\sigma_2^j(K_1,L_1))
\end{align*}
is a bijection such that for all $(i,j) \in \{0,\dots, R_1-1\} \times \{0,\dots,R_2-1\}$ and all $(K,L) \in k^{\N}\times \ell^{\N}$ with $\tau((i,j))=(K,L)$, we have
\[
\tau((i+1,j)) = (\sigma_1(K),L), \text{ and } \tau((i,j+1)) = (K,\sigma_2(K_1,L)),
\]
whenever these are defined.
\end{lem}

\noindent Lemma \ref{lem:existsubset} states that each element of $D$ encodes a finite piece of the grid. We now compute this grid in an example.

\begin{ex}\label{ex:D}
By Example \ref{ex:re} there are $n\in \N$ and natural numbers $m_0<m_1<\dots < m_{11}$ such that
\begin{equation}\label{eq:D1}
\begin{split}
 S(k^n)  \cap \big[\ell^{m_0},\ell^{m_{11}}\big] &= \{ \ell^{m_0}, \ell^{m_3}, \ell^{m_6}, \ell^{m_{11}}\},\\
S(k^{n+1})  \cap \big[\ell^{m_0},\ell^{m_{11}}\big] &= \{ \ell^{m_1}, \ell^{m_4}, \ell^{m_5}\},\\
S(k^{n+2})  \cap \big[\ell^{m_0},\ell^{m_{11}}\big] &= \{ \ell^{m_2}, \ell^{m_7}\},\\
S(k^{n+3})  \cap \big[\ell^{m_0},\ell^{m_{11}}\big] &= \{ \ell^{m_8}, \ell^{m_9}, \ell^{m_{10}}\}.
\end{split}
\end{equation}
Observe that $(k^n,k^{n+3},\ell^{m_0},\ell^{m_{11}}) \in D$. Moreover, 
\[
\card(k^{\N}\cap (k^n,k^{n+3}] = \card(\{k^{n+1},k^{n+2},k^{n+3}\}) = 3 =\card(\{\ell^{m_0}, \ell^{m_3}, \ell^{m_{6}}\} )=\card(S(k^n \cap [\ell^{m_0},\ell^{m_{11}})).
\]
Hence $(k^n,k^{n+3},\ell^{m_0},\ell^{m_{11}}) \in D(3,3)$. Note that
\begin{equation}\label{eq:D2}
\sigma_2^1(k^n,\ell^{m_0}) = \ell^{m_3},\ \sigma_2^2(k^n,\ell^{m_0}) = \ell^{m_6},\ \sigma_2^3(k^n,\ell^{m_0}) = \ell^{m_{11}}.
\end{equation}
Thus in this setting, the map $\tau$ in Lemma \ref{lem:existsubset} is a bijection between $\{0,1,2\}^2$ and 
\[
\{k^{n+1},k^{n+2},k^{n+3}\} \times \{\ell^{m_0},\ell^{m_3},\ell^{m_6}\},
\]
preserving the relevant successor functions. This is visualized in Figure \ref{fig:gridexample}.
\end{ex}

\noindent We have seen how elements of $D$ encode fragments of $\N^2$. The next step is to encode tuples of subsets of such fragments. Here we essentially combine the encoding from the previous subsection with the bijection $\tau$.

\begin{defn} For $c\in \N$, let $\Theta_{c}$ be the set of all triples $(K_1,K,L)\in k^{\N} \times k^{\N} \times \ell^{\N}$ such that 
    \begin{enumerate}
        \item $L\in S(K_1)$ and
    \item $    \card \{ L'\in S(K) \ : \ L < L' < \sigma_2(K_1,L)\} = c.$
    \end{enumerate}
Let $u,v\in \N$ be such that $v\leq u$. Define $\Omega_{u,v}$ to be the
    set of all triples $(K_1,K,L)\in k^{\N} \times k^{\N} \times \ell^{\N}$ such that  there is $c\in \N$ with $c< 2^u$ such that 
\begin{enumerate}
        \item $2^{v-1}$ appears in the binary representation of $c$, and
        \item $(K_1,K,L) \in \Theta_{c}$.
\end{enumerate}
\end{defn}

\noindent Note that for each $c$, the set $\Theta_{c}$ is definable in $(\N,+,k^{\N},\ell^{\N})$. Since $u$ is fixed and there are only finitely many numbers smaller than $2^u$, the set $\Omega_{u,v}$ is definable in $(\N,+,k^{\N},\ell^{\N})$. For $K_1 \in k^{\N}$, we write $\Omega_{u,v}(K_1)$ for the set of all $(K,L)\in k^{\N} \times \ell^{\N}$ such that $(K_1,K,L)\in \Omega_{u,v}$.

\begin{ex} We continue with the notation from Example \ref{ex:D}. In particular, we will be working with the tuple $(k^n,k^{n+3},\ell^{m_0},\ell^{m_{11}}) \in D$. We will explain how this tuple encodes a pair of subsets of $\N^2.$  Using \eqref{eq:D1} and \eqref{eq:D2}, we first observe that
$(k^n,k^{n+1},\ell^{m_0})\in \Theta_1$, because
\[
\{ L'\in S(k^n) \ : \ \ell^{m_0} < L' < \sigma_2(k^n,\ell^{m_0})\} =  \{ L'\in S(k^n) \ : \ \ell^{m_0} < L' <\ell^{m_3})\} = \{\ell^{m_2}\}.
\]
Similarly, we obtain $(k^n,k^{n+1},\ell^{m_3})\in \Theta_2$, since 
\[
\{ L'\in S(k^n) \ : \ \ell^{m_3} < L' < \sigma_2(k^n,\ell^{m_3})\} =  \{ L'\in S(k^n) \ : \ \ell^{m_3} < L' <\ell^{m_6})\} = \{\ell^{m_4},\ell^{m_5}\}.
\]
Indeed, it is easy to see, again using \eqref{eq:D1} and \eqref{eq:D2}, that
\begin{align*}
\{ (i,j)\in \{0,1,2\}^2 \ : \ (k^n,k^{n+i+1},\sigma_2^{j}(k^n,\ell^{m_0})) \in \Theta_1\} &= \{(0,0),(1,0),(1,2)\},\\
\{ (i,j)\in \{0,1,2\}^2 \ : \ (k^n,k^{n+i+1},\sigma_2^{j}(k^n,\ell^{m_0})) \in \Theta_2\} &= \{(0,1)\},\\
\{ (i,j)\in \{0,1,2\}^2 \ : \ (k^n,k^{n+i+1},\sigma_2^{j}(k^n,\ell^{m_0})) \in \Theta_3\} &= \{(2,2)\}.
\end{align*}
It is worth checking how this corresponds to the rows and columns marked with an X in Figure \ref{fig:patternexample}. For example, the fact that $(2,2)$ is in the third set, is due to the three columns marked with an X in the fourth row between the seventh and eleventh column.\newline
Note that the above subsets of $\N^2$ are necessarily disjoint. In order to encode arbitrary subsets of $\N^2$, we need to use $\Omega$. Before doing so, we choose the length $u$ of the tuple of subsets. For this example, it is natural to choose $u=2$, although any other choice would result in a valid tuple, too. Then we compute
\begin{align*}
   X_1 := \{ (i,j) \in \{0,1,2\}^2: (k^{n+i+1},\sigma_2^{j}(k^n,\ell^{m_0})) \in \Omega_{2,1}(k^n)\} &= \{(0,0),(1,0),(1,2),(2,2)\}\\
   X_2 := \{ (i,j) \in \{0,1,2\}^2: (k^{n+i+1},\sigma_2^{j}(k^n,\ell^{m_0})) \in \Omega_{2,2}(k^n)\} &= \{(0,1),(2,2)\}.
\end{align*}
Therefore the tuple $(k^n,k^{n+3},\ell^{m_0},\ell^{m_{11}})$ together with our choice $u=2$ encodes the pair $(X_1,X_2)$ of subsets of $\N^2$.
\end{ex}

\noindent In the following lemmas we prove that every tuple of subsets of $\{0,1,\dots,R-1\}^2$ is encoded this way as some element of $D(R,R)$.

\begin{lem}\label{lem:member}
Let $X=(X_1,\dots, X_u) \subseteq \mathcal{P}(\N^2)^u$ and $R\in \N$. Then there is $(K_1,K_2,L_1,L_2)\in D(R,R)$ such that  for all $i,j \in \{0,\dots,R-1\}$ and $v\in \{1,\dots, u\}$  
\[
(i,j) \in X_v \hbox{ if and only if }(\sigma_1^{i+1}(K_1),\sigma_2^{j}(K_1,L_1)) \in \Omega_{u,v}(K_1).
\]
\end{lem}
\begin{proof}
Let $s\in \N$ and $r=r_0\cdots r_s\in \{0,1,\dots, R\}^*$ be such that $w(B(X,R))=r.$ We apply Lemma \ref{lem:pickoutvalues} to $R,s$ and $r$. We obtain $t \in \mathbb{N}_{>0}$ and an increasing sequence $m_0<m_1<\dots<m_s$  of natural numbers such that for all $i\in\{0,\dots,R\}$
\begin{equation}\label{eq:choice}
        S(k^{t + i})\cap \big[\ell^{m_0},\ell^{m_s}\big] =\{ \ell^{m_j} \ : \  j \in \{1,\dots, s\} \wedge r_j = i\}.
\end{equation}
We now set $K_1 = k^t, K_2 = k^{t+R}, L_1 = \ell^{m_0},$ and  $L_2 = \ell^{m_s}$.\newline

\noindent We first show that $(K_1,K_2,L_1,L_2)\in D(R,R)$.
Observe that for $i\in \{0,\dots,R\}$, we have $\sigma_1^i(K_1) = k^{t+i}$.  Hence we obtain from \eqref{eq:choice} that
\begin{equation*}
S(K_1) \cap [\ell^{m_0},\ell^{m_s}] = \{ \ell^{m_i} \ : \ i \in \{0,\dots,s\} \wedge r_i =0\}.  
\end{equation*}
Thus, in particular, for each $i \in \{0,\dots,R\}$
\begin{equation}\label{eq:mzrj}
\sigma_2^i(K_1,L_1)= \ell^{m_{Z(r,i+1)}}.
\end{equation}
Since $r_1=r_s=0$, we have that $L_1,L_2 \in S(K_1).$ Thus $(K_1,K_2,L_1,L_2)\in D$. 
Since $0$ appears $(R+1)$-times in $r$, we have that
\[
\card \big(S(K_1) \cap [\ell^{m_1},\ell^{m_s})\big) = R,
\]
and hence $(K_1,K_2,L_1,L_2) \in D(R,R)$.\newline

\noindent Let $i,j \in \{0,\dots,R-1\}$. We now show that
\begin{equation}\label{eq:crij}
C(r,i+1,j+1) = \card \{ L' \in S(\sigma_1^{i+1}(K_1)) \ : \ \sigma_2^{j}(K_1,L_1) < L' < \sigma_2(K_1,\sigma_2^{j}(K_1,L_1))\}.
\end{equation}
Then the equivalence in the statement of the lemma for our choice of $(K_1,K_2,L_1,L_2)$ follows from the definition of $\Theta_{c}$ and $\Omega_{u,v}$ and Corollary \ref{cor:count}.\newline

\noindent We now establish \eqref{eq:crij}. 
From \eqref{eq:choice} and \eqref{eq:mzrj} we deduce
\begin{align*}
\{ L' \in S(\sigma_1^{i+1}(K_1)) \ &: \ \sigma_2^j(K_1,L_1) < L' < \sigma_2(K_1,\sigma_2^j(K_1,L_1))\} \\&= S(k^{t+i+1}) \cap (\ell^{m_{Z(r,j+1)}},\ell^{m_{Z(r,j+2)}}) \\
&= \{ \ell^{m_p} : Z(r,j+1) < p < Z(r,j+2) \wedge r_p = i+1\}.
\end{align*}
The cardinality of the latter set can be computed as:
\begin{align*}
\card \{ \ell^{m_p} : Z(r,j+1) &< p < Z(r,j+2) \wedge r_p = i+1\}\\ &= \card \{ p \ : \ Z(r,j+1) < p < Z(r,j+2) \wedge r_p = i+1\}\\ &= C(r,i+1,j+1).
\end{align*}
This proves \eqref{eq:crij}.
\end{proof}
 


\begin{proof}[Proof of Theorem B]
Let $\mathcal{M}$ be a Turing machine and let $\mathcal{L}$ be the first-order language of \linebreak $(\N,+,k^{\N},\ell^{\N})$. In the following, we find an $\mathcal{L}$-sentence $\varphi_{\mathcal{M}}$ such that $\mathcal{M}$ halts when started on the empty tape if and only if $(\N,+,k^{\N},\ell^{\N})\models \varphi_{\mathcal{M}}$. Our construction of $\varphi_{\mathcal{M}}$ is effective. Thus the undecidability of the theory of $(\N,+,k^{\N},\ell^{\N})$ follows from the undecidability of the halting problem.\newline

\noindent Let $\mathcal{M}=(Q,\Sigma,\delta,q_0,q_{accept})$. Without loss of generality, we can assume that $Q=\{1,\dots,u_1\}$, $\Sigma=\{1,\dots,u_2\}$, $q_0=1$ and $q_{accept}=2$, and the transition function $\delta=(\delta_1,\delta_2,\delta_3)$ is a total function from $[u_1]\times [u_2]$ to $[u_1] \times [u_2]\times \{-1,1\}$. \newline

\noindent It is well-known that subsets of $\N^2$ can be used to code a computation on $\mathcal{M}$. Here we follow the argument in the proof of \cite[Theorem 16.5]{Thomas} and its presentation in the proof of \cite[Theorem 7.1]{hnp}. Let $A_1,\dots,A_{u_1}, B_1,\dots B_{u_2} \subseteq \N^2$. The $A_i$'s are used to store the current state of $\mathcal{M}$; that is, for $(s,t)\in \N^2$, we have $(s,t) \in A_i$ if and only if at step $s$ of the computation, $\mathcal{M}$ is in state $i$ and its head is over the $t$-th cell of the tape. Similarly, the $B_j$'s code the symbols written on the tape: $(s,t)\in B_j$ if and only if $j$ is written on the the $t$-th cell of the tape at step $s$ of the computation. Thus $\mathcal{M}$ halts when started on the empty tape if and only if there are finite  $A_1,\dots,A_{u_1},B_1,\dots B_{u_2} \subseteq \N^2$ and $R_1,R_2 \in \N$ such that
\begin{enumerate}
    \item the $A_i$'s are pairwise disjoint subsets of $\{0,\dots,R_1-1\}\times \{0,\dots,R_2-1\}$ and the $B_j$'s are pairwise disjoint of $\{0,\dots,R_1-1\}\times \{0,\dots,R_2-1\}$,
    \item $(0,0)\in A_1$; i.e.\ the computation starts in the initial state and with the head over the $0$-th cell,
    \item $(0,t) \notin \bigcup_{j=1}^{u_2} B_j$ for all $t\in \{0,\dots,R_2-1\}$; i.e.\ $\mathcal{M}$ is started on the empty tape,
    \item $A_2\neq \emptyset$; i.e.\ $\mathcal{M}$ eventually halts,
    \item for each $s\in \{0,\dots,R_1-1\}$, there is at most one $t\in \{0,\dots,R_2-1\}$ such that $(s,t)\in \bigcup_{i=1}^{u_1} A_i$; i.e.\ at each step of the computation, $\mathcal{M}$ is in exactly one state,
    \item If $(s,t) \in B_j$, then
    \begin{enumerate}
        \item if $(s,t) \notin \bigcup_{i=1}^{u_1} A_i$, then $(s+1,t)\in B_j$; i.e.\ if the head is not over the $t$-th cell, then the symbol in this cell does not change,
        \item if $(s,t)\in A_i$, then $(s+1,t) \in B_{\delta_2(i,j)}$ and $(s+1,t+\delta_3(i,j)) \in A_{\delta_1(i,j)}$; i.e.\ if the head is over the $t$-th cell and $\mathcal{M}$ is in state $i$, then the appropriate transition rule is applied.
    \end{enumerate}
\end{enumerate}

\noindent It is left to construct a $\mathcal{L}$-sentence $\varphi_{\mathcal{M}}$ such that $(\N,+,k^{\N},\ell^{\N})\models \varphi_{\mathcal{M}}$ if and only if there are $A_1,\dots,A_{u_1}, B_1,\dots B_{u_2} \subseteq \N^2$ and $R_1,R_2\in \N$ satisfying (1)-(6). Set $u=u_1+u_2$. Keeping Lemma \ref{lem:member} in mind, we define an $\mathcal{L}$-formula $\theta_{\mathcal{M}}(K_1,K_2,L_1,L_2)$ as follows:
\begin{align}
 \forall K &\in k^{\N} \forall L \in S(K_1) (K_1< K \leq K_2 \wedge L_1 \leq L < L_2) \rightarrow \notag \\
&\Bigg[\bigwedge_{1\leq i,j \leq u_1, i\neq j} (K_1,K,L) \notin \Omega_{u,i} \cap \Omega_{u,j}  \wedge \bigwedge_{u_1 < i,j \leq u, i\neq j} (K_1,K,L) \notin \Omega_{u,i} \cap \Omega_{u,j} \raisetag{-25pt}\tag{a}\\
&\qquad \wedge (K_1,\sigma_1(K_1),L_1) \in \Omega_{u,1}  \wedge \bigwedge_{j=1}^{u_2} (K_1,\sigma_1(K_1),L) \notin \Omega_{u,u_1+j}\raisetag{-20pt}\tag{b}\\
&\qquad \wedge \exists K' \exists L'\in S(K_1) \ K_1<K'\leq K_2 \wedge (K_1,K',L') \in \Omega_{u,2} \raisetag{-15pt} \tag{c}\\
&\qquad \wedge \Big( \exists L' \in S(K_1) \ \bigvee_{i=1}^{u_1} (K_1,K,L') \in \Omega_{u,i} \wedge \forall L'' \in S(K_1) (\bigvee_{i=1}^{u_1} (K_1,K,L'') \in \Omega_{u,i}) \rightarrow L'=L''\Big)\raisetag{-25pt}\tag{d}\\
& \qquad \wedge \bigwedge_{j=1}^n (K_1,K,L) \in \Omega_{u,u_1+j} \rightarrow  \bigg( \Big(\big(\bigwedge_{i=1}^{u_1}  (K_1,K,L) \notin \Omega_{u,i}\big) \rightarrow (K_1,\sigma_1(K),L) \in \Omega_{u,u_1+j}\Big) \notag\\
&\qquad \qquad \qquad \wedge \bigwedge_{i=1}^{u_1}  \Big[(K_1,K,L) \in \Omega_{u,i} \rightarrow 
\Big((K_1,\sigma_1(K),L) \in \Omega_{u,u_1+\delta_2(i,j)} \raisetag{-23pt} \tag{e}\\
& \qquad \qquad \qquad \qquad \qquad \wedge (K_1,\sigma_1(K),\sigma_2^{\delta_3(i,j)}(K_1,L)) \in \Omega_{u,\delta_1(i,j)} \Big) \Big]\bigg) \Bigg] \notag
\end{align}
Note that statement (a) in the display above corresponds to statement (1), statement (b) to (2) and (3), statement (c) to (4), statement (d) to (5), and finally statement (e) to (6).  
Set 
\[
\varphi_{\mathcal{M}}:= \exists (K_1,K_2,L_1,L_2) \in D \ \theta_{\mathcal{M}}(K_1,K_2,L_1,L_2).
\]
We now show that $\varphi_{\mathcal{M}}$ has the desired property.\newline

\noindent Suppose $A_1,\dots,A_{u_1}, B_1,\dots B_{u_2} \subseteq \N^2$ and $R_1,R_2 \in \N$ satisfy
(1)-(6). Set $R:= \max\{R_1,R_2\}$. By Lemma \ref{lem:member} there exist $(K_1,K_2,L_1,L_2) \in D(R,R)$ such that for $i,j,v_1,v_2\in \N$ with $i,j< R$,  $v_1\leq u_1$ and $v_2\leq u_2$
\begin{itemize}
    \item $(i,j) \in A_{v_1}$ if and only if $(K_1,\sigma_1^{i+1}(K_1),\sigma_2^j((K_1,L_1)) \in \Omega_{u,v_1}$,
    \item $(i,j) \in B_{v_2}$ if and only if $(K_1,\sigma_1^{i+1}(K_1),\sigma_2^j((K_1,L_1)) \in \Omega_{u,u_2+v_2}$,
\end{itemize}
and for all $K \in k^{\N}$ and $L\in S(K_1)$ with $K_1\leq K \leq K_2 \wedge L_1 \leq L \leq L_2$ there are $i,j\in \N$ with $i,j\leq R$ such that $K = \sigma_1^i(K_1) \text{ and } L = \sigma_2^j((K_1,L_1))$. Thus our choice of $(K_1,K_2,L_1,L_2)$ satisfies $\theta_{\mathcal{M}}$, and hence $(\N,+,k^{\N},\ell^{\N})\models \varphi_{\mathcal{M}}$.\newline

\noindent Now suppose that $(\N,+,k^{\N},\ell^{\N})\models \varphi_{\mathcal{M}}.$ Then there is a tuple $(K_1,K_2,L_1,L_2)\in D$ such that \linebreak $\theta_{\mathcal{M}}(K_1,K_2,L_1,L_2)$ holds. Set
\[
 R_1:= \card \big(k^{\N} \cap (K_1,K_2]\big), \ R_2:= \card \big(S(K_1) \cap [L_1,L_2)).
\]
Let $(X_1,\dots,X_{u})\in \mathcal{P}(\{0,\dots,R_1-1\} \times \{0,\dots,R_2-1\})^{u}$ be such that
for all $i \in \{0,\dots,R_1-1\}$, $j \in  \{0,\dots,R_2-1\}$ and $v\in \{1,\dots, u\}$
\[
(i,j) \in X_v \text{ if and only if } (\sigma_1^{i+1}(K_1),\sigma_2^j(K_1,L_1)) \in \Omega_{u,v}(K_1).
\]
For each  $v=1,\dots,u_1$, set $A_v := X_v$, and for each $v=1,\dots,u_2$, set $B_u:= X_{u_1+v}$. Using Lemma \ref{lem:existsubset} and the fact that $\theta_{\mathcal{M}}(K_1,K_2,L_1,L_2)$ holds, we easily check that $A_1,\dots,A_{u_1},B_1,\dots,B_{u_2}$ and $R_1,R_2$ satisfy (1)-(6).
\end{proof}

\section{Model-theoretic classification}

Theorem B establishes that from a computational point of view, the theory ${\rm FO}(\mathbb{N},+,k^{\N},\ell^{\N})$ is intractable. Given that by \cite{S-undefinability} the structure $(\mathbb{N},+,k^{\N},\ell^{\N})$ does not define multiplication, the precise model-theoretic complexity of this structure is not immediately clear. Logicians, and foremost Shelah, have developed many different combinatorical tameness notions to capture various degrees of model-theoretic complexity. Here we consider the following class of properties introduced by Shelah \cite{Shelah-Strongly, Shelah-2dep}.

\begin{defn}
Let $T$ be a first-order theory. We say $T$ has the \textbf{$k$-independence property} if there is a formula $\varphi(x,y_1,\dots, y_{k})$ in the language of $T$ such that for every $n\in \N$ there is a model $M$ of $T$ and elements $a_{i,j}$ for $i\in \{1,\dots,k\}$ and $j\in \{0,\dots,n-1\}$
 such that for every $S \subseteq \{0,\dots,n-1\}^k$ there is $b_S\in M$ such that
\[
M \models \varphi(b_S,a_{1,j_1},\dots,a_{k,j_k}) \text{ if and only if } (j_1,\dots,j_k) \in S.
\]
\end{defn}

\noindent When $k=1$, this is better known as the independent property and abbreviated by \textbf{IP}. It is immediate from the definition that the $(k+1)$-independence property implies the $k$-independence property. It is important to note that the \emph{absence} of the $k$-independence property, in particular the absence of the 1-independence property, indicates tameness, while its presence provides evidence of complexity. For example, if $T$ does not have the 1-independence property, then every definable family has finite Vapnik-Chervonenkis dimension. For other consequences of the absence of the 1-independence property (this property is usually called \textbf{NIP}), see Simon \cite{Simon-Book}, and more generally of the absence of the $k$-independence property, see Chernikov, Palacin, Takeuchi \cite{CPT-n-dependence}. Our result here is the following.

\begin{thm}\label{thm:2ip}
Let $k,\ell\in \N_{\geq 2}$ be multiplicatively independent. Then the theory ${\rm FO}(\mathbb{N},+,k^{\N},\ell^{\N})$ has the $2$-independence property.
\end{thm}
\begin{proof}
Let $\varphi(K_1, K_2, L_1, L_2, x, y)$ be the formula in the language ${\rm FO}(\mathbb{N},+,k^{\N},\ell^{\N})$ stating that $x \in k^\mathbb{N}$ with $K_1 \leq y \leq K_2$, $y \in \ell^\mathbb{N}$ with $L_1 \leq y \leq L_2$, and $(x, y) \in \Omega_{1,1}(K_1)$. Then it is immediate from Lemma \ref{lem:member} that if $X \subseteq \mathcal{P}(\mathbb{N}^2)$ and $R \in \mathbb{N}$, we have that for all $i, j \in \mathbb{N}$ with $i, j \in \{0,\dots,R-1\}$, 
\[
(i, j) \in X \text{ if and only if }\varphi(K_1, K_2, L_1, L_2, \sigma_1^{i+1}(K_1),\sigma_2^{j}(K_1,L_1)). 
\]
Thus the formula $\varphi$ witnesses that ${\rm FO}(\mathbb{N},+,k^{\N},\ell^{\N})$ has the $2$-independence property.
\end{proof}

\section{Conclusion and open questions}

\subsection{Different domains} The proof of Theorem B does not only work over $\N$. 

\begin{prop}\label{prop:reals} Let $k,\ell \in \N_{\geq 2}$ be multiplicatively independent. Then the theory \linebreak ${\rm FO}(\mathbb{R},<,+,k^{\N},\ell^{\N})$ is undecidable, but $(\mathbb{R},<,+,k^{\N},\ell^{\N})$ does not define multiplication on $\R$ (even when parameters are allowed).
\end{prop}
\begin{proof}
 The reader can check that the same argument as in the proof of Theorem B shows the undecidability of ${\rm FO}(\mathbb{R},<,+,k^{\N},\ell^{\N})$. Let $(\R,<,+,\Z)^{\#}$ be the expansion of $(\R,<,+,\Z)$ by predicates for every subset $X$ of $\Z^n$ for every $n\in \N$. Clearly, every set definable with parameters in  $(\mathbb{R},<,+,k^{\N},\ell^{\N})$ is also definable with parameters in $(\R,<,+,\Z)^{\#}$. However, by \cite[Theorem A]{FM} the structure $(\R,<,+,\Z)^{\#}$ does not define multiplication even with parameters (see \cite[Example on p.56]{FM}). 
\end{proof}

\noindent The argument for undecidability also extends unchanged to  ${\rm FO}(\mathbb{Q},<,+,k^{\N},\ell^{\N})$, and it is natural to ask for which other domains our argument goes through. Until now, undecidability was only known for the theory $\mathrm{FO}(\mathbb{R},<,+,\cdot,k^{\N},\ell^{\N})$ by \cite[Theorem 1.3]{H2010} and, more generally, for the theory $\mathrm{FO}(K,<,+,\cdot,k^{\N},\ell^{\N})$, where $K$ is a subfield of $\R$, by \cite[Theorem C]{H2021}. By Propostion \ref{prop:reals} the structure $(\mathbb{R},<,+,\cdot,k^{\N},\ell^{\N})$ is strictly more expressive than $(\mathbb{R},<,+,k^{\N},\ell^{\N})$.\newline

\noindent Observe that above we always added order to the structure when changing the domain. In $(\N,+)$ the order relation is definable, but this is no longer true even when replacing $\N$ by $\Z$. And, indeed, when we consider these other domains without order, we see very different behaviors. In contrast to Theorem \ref{thm:2ip} the theory $\mathrm{FO}(\Z,+,k^{\N},\ell^{\N})$ does \emph{not} even have the 1-independence property by Conant \cite{Conant-Subgroups}. Hence it should be considered model-theoretically tame. Even more is true: By Wang \cite{Wang-thesis} the theory $\mathrm{FO}(\Z,+,k^{\N},\ell^{\N})$ is decidable if and only if we can compute for all $a_1,\dots,a_n\in \Q$ the set of all tuples $(z_1,\dots,z_n) \in (k^{\Z}\ell^{\Z})^n$ such that
\[
a_1z_1 + \dots + a_nz_n = 1, \hbox{ and } \sum_{i \in J} a_i z_i \neq 0 \hbox{ for all $J\subsetneq \{1,\dots,n\}$.}
\]
While it is known that these sets are always finite for all $a_1,\dots,a_n\in \Q$, it is an open question whether they are computable. However, an effective version of the $p$-adic subspace theorem by Schlickewei \cite{Schlickewei-padic}, which is an open problem in number theory, would imply the desired computability. It is an interesting question whether these methods could be used to show the decidability of fragments of $\mathrm{FO}(\N,+,k^{\N},\ell^{\N})$. It seems likely that the existential fragment of this theory is indeed decidable under the same assumption as in \cite{Wang-thesis}.

\subsection{Other numeration systems} Cobham's theorem has been extended from classical base-$k$ numeration to other non-classical numeration systems, in particular by B\`es \cite{Bes-CS}. See \cite{durandrigo} for a survey. We call a strictly increasing sequence $U=(U_n)_{n\in \N}$ of natural numbers a \textbf{numeration system} if $U_0=1$ and $\{U_{n+1}/U_n \ : \ n \in \N\}$ is bounded. We say such a system $U$ is \textbf{linear} if there are $d_0,\dots,d_{k-1}\in \Z$ such that for every $n\in \N_{\geq k}$
\[
U_n = d_{k-1}U_{n-1} + \dots + d_0 U_{n-k}.
\]
The polynomial $P_U=X^k - d_{k-1}X^{n-1} - \dots - d_1 X - d_0$ is the \textbf{characteristic polynomial} of $U$. 
There is a natural extension of $k$-recognizability to $U$-recognizability of subsets of $\N^n$ (see for example \cite{Bes-CS}). Let $\theta, \theta'$ be multiplicatively independent Pisot numbers, and let $U,U'$ be two linear numeration systems such that $P_U$ is the minimal polynomial of $\theta$ and $P_{U'}$ is the minimal polynomial of $\theta'$. By \cite[Theorem 3.1]{Bes-CS}, whenever $X\subseteq \N^n$ is both $U$- and $U'$-recognizable, then $X$ is definable in $(\N,+)$. It is natural to ask whether Theorem A extends as well. 
\begin{ques}\label{quest:nonstandard}
Let $X\subseteq \N^m$ be $U$-recognizable, let $Y\subseteq \N^n$ be $U'$-recognizable such that $X,Y$ are not definable in $(\N,+)$. Is the theory $\mathrm{FO}(\N,+,X,Y)$ undecidable? 
\end{ques}
\noindent Let $X\subseteq \N^m$ be $U$-recognizable and not definable in $(\N,+)$. As far as we know, it is even open whether the analogue of Fact \ref{fact:bes} holds in this setting; that is, whether $(\N,+,X)$ defines $\{ U_n \ : \ n\in\N\}$. However, an analog of Fact \ref{fact:BHMV} is known. Let $V_U: \N \to \N$ be the function that maps $0$ to $1$ and, if $x\geq 1$, then $V_U(x)$ is the least $U_i$ appearing the normalized $U$-representation of $x$ with non-zero coefficient. Then by Bruy\`ere and Hansel \cite{BH-Betrand} the theory $\mathrm{FO}(\N,+,V_U)$ is decidable, and, $X$ is $U$-recognizable if and only if $X$ is definable in $(\N,+,V_U)$.
Thus also in this case, a positive answer to Question \ref{quest:nonstandard} implies the appropriate version of the Cobham-Sem\"enov theorem.

\subsection{Other related work} There is recent related work on undecidability in expansions of $(\N,+,V_k)$ by B\`es \cite{BesCard}, and earlier work by Elgot and Rabin \cite{ElgotRabin} and Thomas \cite{ThomasUnde}. In \cite{PRW} Point, Rigo and Waxweiler adapt \cite{Villemaire} to the additive reduct of certain Euclidean rings. It is worth asking whether our new method also works in this more general setting. Another related recent work is Haase and R\'{o}zycki \cite{HasseRoz}. There is also a new and very short proof of Cobham's theorem by Krebs \cite{Krebs}, although we don't see an immediate connection to our work.

\bibliographystyle{amsplain}
\bibliography{biblio}

\providecommand{\bysame}{\leavevmode\hbox to3em{\hrulefill}\thinspace}
\providecommand{\MR}{\relax\ifhmode\unskip\space\fi MR }
\providecommand{\MRhref}[2]{%
  \href{http://www.ams.org/mathscinet-getitem?mr=#1}{#2}
}
\providecommand{\href}[2]{#2}
\begin{thebibliography}{10}

\bibitem{AS}
Jean-Paul Allouche and Jeffrey Shallit, \emph{Automatic sequences}, Cambridge
  University Press, Cambridge, 2003, Theory, applications, generalizations.
  \MR{1997038}

\bibitem{Apostol}
Tom~M. Apostol, \emph{Modular functions and {D}irichlet series in number
  theory}, second ed., Graduate Texts in Mathematics, vol.~41, Springer-Verlag,
  New York, 1990. \MR{1027834}

\bibitem{Baker}
A.~Baker, \emph{Linear forms in the logarithms of algebraic numbers. {I}, {II},
  {III}}, Mathematika \textbf{13} (1966), 204--216; ibid. 14 (1967), 102--107;
  ibid. 14 (1967), 220--228. \MR{220680}

\bibitem{Bes97}
Alexis B\`es, \emph{Undecidable extensions of {B}\"{u}chi arithmetic and
  {C}obham-{S}em\"{e}nov theorem}, J. Symbolic Logic \textbf{62} (1997), no.~4,
  1280--1296. \MR{1617949}

\bibitem{Bes-CS}
\bysame, \emph{An extension of the {C}obham-{S}em\"{e}nov theorem}, J. Symbolic
  Logic \textbf{65} (2000), no.~1, 201--211. \MR{1782115}

\bibitem{Bes-Survey}
\bysame, \emph{A survey of arithmetical definability}, no. suppl., 2001, A
  tribute to Maurice Boffa, pp.~1--54. \MR{1900397}

\bibitem{BesCard}
\bysame, \emph{Expansions of {MSO} by cardinality relations}, Log. Methods
  Comput. Sci. \textbf{9} (2013), no.~4, 4:18, 17. \MR{3145056}

\bibitem{BH-Betrand}
V\'{e}ronique Bruy\`ere and Georges Hansel, \emph{Bertrand numeration systems
  and recognizability}, vol. 181, 1997, Latin American Theoretical INformatics
  (Valpara\'{\i}so, 1995), pp.~17--43. \MR{1463527}

\bibitem{BHMV}
V\'{e}ronique Bruy\`ere, Georges Hansel, Christian Michaux, and Roger
  Villemaire, \emph{Logic and {$p$}-recognizable sets of integers}, vol.~1,
  1994, Journ\'{e}es Montoises (Mons, 1992), pp.~191--238. \MR{1318968}

\bibitem{Buechi60}
J.~Richard B\"{u}chi, \emph{Weak second-order arithmetic and finite automata},
  Z. Math. Logik Grundlagen Math. \textbf{6} (1960), 66--92. \MR{125010}

\bibitem{CPT-n-dependence}
Artem Chernikov, Daniel Palacin, and Kota Takeuchi, \emph{On {$n$}-dependence},
  Notre Dame J. Form. Log. \textbf{60} (2019), no.~2, 195--214. \MR{3952231}

\bibitem{Cobham}
Alan Cobham, \emph{On the base-dependence of sets of numbers recognizable by
  finite automata}, Math. Systems Theory \textbf{3} (1969), 186--192.
  \MR{250789}

\bibitem{Conant-Subgroups}
Gabriel Conant, \emph{Multiplicative structure in stable expansions of the
  group of integers}, Illinois J. Math. \textbf{62} (2018), no.~1-4, 341--364.
  \MR{3922420}

\bibitem{durandrigo}
Fabien Durand and Michel Rigo, \emph{On {C}obham's theorem}, Handbook of
  automata theory. {V}ol. {II}. {A}utomata in mathematics and selected
  applications, EMS Press, Berlin, 2021, pp.~947--986. \MR{4380949}

\bibitem{ElgotRabin}
Calvin~C Elgot and Michael~O Rabin, \emph{Decidability and undecidability of
  extensions of second (first) order theory of (generalized) successor}, J.
  Symbolic Logic \textbf{31} (1966), no.~2, 169--181.

\bibitem{FM}
Harvey Friedman and Chris Miller, \emph{Expansions of o-minimal structures by
  sparse sets}, Fund. Math. \textbf{167} (2001), no.~1, 55--64. \MR{1816817}

\bibitem{GinsSpanier}
Seymour Ginsburg and Edwin~H. Spanier, \emph{Semigroups, {P}resburger formulas,
  and languages}, Pacific J. Math. \textbf{16} (1966), 285--296. \MR{191770}

\bibitem{Hasse-Survey}
Christoph Haase, \emph{A survival guide to presburger arithmetic}, ACM SIGLOG
  News \textbf{5} (2018), no.~3, 67--82.

\bibitem{HasseRoz}
Christoph Haase and Jakub R\'{o}\.{z}ycki, \emph{On the expressiveness of
  {B}\"{u}chi arithmetic}, Foundations of software science and computation
  structures, Lecture Notes in Comput. Sci., vol. 12650, Springer, Cham, 2021,
  pp.~310--323. \MR{4240451}

\bibitem{H2010}
Philipp Hieronymi, \emph{Defining the set of integers in expansions of the real
  field by a closed discrete set}, Proc. Amer. Math. Soc. \textbf{138} (2010),
  no.~6, 2163--2168. \MR{2596055}

\bibitem{H2021}
\bysame, \emph{Expansions of subfields of the real field by a discrete set},
  Fund. Math. \textbf{215} (2011), no.~2, 167--175. \MR{2860183}

\bibitem{hnp}
Philipp Hieronymi, Danny Nguyen, and Igor Pak, \emph{Presburger arithmetic with
  algebraic scalar multiplications}, Log. Methods Comput. Sci. \textbf{17}
  (2021), no.~3, Paper No. 4, 34 pp. \MR{4298407}

\bibitem{Krebs}
Thijmen J.~P. Krebs, \emph{A more reasonable proof of {C}obham's theorem},
  Internat. J. Found. Comput. Sci. \textbf{32} (2021), no.~2, 203--207.
  \MR{4218824}

\bibitem{mcnaughton_1963}
Robert McNaughton, \emph{J. {R}ichard {B}\"uchi. weak second-order arithmetic
  and finite automata}, J. Symbolic Logic \textbf{28} (1963), no.~1, 100--102.

\bibitem{MV-Survey}
Christian Michaux and Roger Villemaire, \emph{Open questions around {B}\"{u}chi
  and {P}resburger arithmetics}, Logic: from foundations to applications
  ({S}taffordshire, 1993), Oxford Sci. Publ., Oxford Univ. Press, New York,
  1996, pp.~353--383. \MR{1428012}

\bibitem{PRW}
Fran\c{c}oise Point, Michel Rigo, and Laurent Waxweiler, \emph{Defining
  multiplication in some additive expansions of polynomial rings}, Comm.
  Algebra \textbf{44} (2016), no.~5, 2075--2099. \MR{3490667}

\bibitem{point2010expansion}
Francoise Point, \emph{On the expansion $(\mathbb{N},+, 2^x)$ of presburger
  arithmetic},  (2007).

\bibitem{Schlickewei-padic}
Hans~Peter Schlickewei, \emph{The {${\mathfrak p}$}-adic
  {T}hue-{S}iegel-{R}oth-{S}chmidt theorem}, Arch. Math. (Basel) \textbf{29}
  (1977), no.~3, 267--270. \MR{491529}

\bibitem{S-undefinability}
Christian Schulz, \emph{Undefinability of multiplication in presburger
  arithmetic with sets of powers}, arXiv:2209.11858 (2022).

\bibitem{Semenov}
A.~L. Semenov, \emph{The {P}resburger nature of predicates that are regular in
  two number systems}, Sibirsk. Mat. \v{Z}. \textbf{18} (1977), no.~2,
  403--418, 479. \MR{0450050}

\bibitem{Shelah-Strongly}
Saharon Shelah, \emph{Strongly dependent theories}, Israel J. Math.
  \textbf{204} (2014), no.~1, 1--83. \MR{3273451}

\bibitem{Shelah-2dep}
\bysame, \emph{Definable groups for dependent and 2-dependent theories},
  Sarajevo J. Math. \textbf{13(25)} (2017), no.~1, 3--25. \MR{3666349}

\bibitem{Simon-Book}
Pierre Simon, \emph{A guide to {NIP} theories}, Lecture Notes in Logic,
  vol.~44, Association for Symbolic Logic, Chicago, IL; Cambridge Scientific
  Publishers, Cambridge, 2015. \MR{3560428}

\bibitem{ThomasUnde}
Wolfgang Thomas, \emph{A note on undecidable extensions of monadic second order
  successor arithmetic}, Arch. Math. Logik Grundlag. \textbf{17} (1975),
  no.~1-2, 43--44. \MR{441712}

\bibitem{Thomas}
\bysame, \emph{Finite automata and the analysis of infinite transition
  systems}, Modern applications of automata theory, IISc Res. Monogr. Ser.,
  vol.~2, World Sci. Publ., Hackensack, NJ, 2012, pp.~495--527. \MR{3024650}

\bibitem{Villemaire}
Roger Villemaire, \emph{The theory of {$\langle {\bf N},+,V_k,V_l\rangle$} is
  undecidable}, Theoret. Comput. Sci. \textbf{106} (1992), no.~2, 337--349.
  \MR{1192774}

\bibitem{Wang-thesis}
Xiaoduo Wang, \emph{Quantifier elimination and decidability of the theory of
  additive integer group augmented by predicates of multiplicative cyclic
  submonoids}, 2022, Thesis (Master)--University of Illinois at
  Urbana-Champaign.

\end{thebibliography}
\end{document}